\documentclass[twoside,a4paper,reqno,11pt]{amsart} 
\usepackage[top=28mm,right=28mm,bottom=28mm,left=28mm]{geometry}
\usepackage{microtype}
\usepackage{amsfonts, amsmath, amssymb, amsbsy, mathrsfs, bm, latexsym, stmaryrd, hyperref, array, enumitem, textcomp, url}
\usepackage[table]{xcolor}

\renewcommand{\a}{\alpha}
\renewcommand{\b}{\beta}

\newcommand{\e}{\epsilon}

\renewcommand{\O}{\Omega}

\newcommand{\la}{\langle}
\newcommand{\ra}{\rangle}
\newcommand{\leqs}{\leqslant}
\newcommand{\geqs}{\geqslant}
\newcommand{\normeq}{\trianglelefteqslant}

\newcommand{\vs}{\vspace{3mm}}

\makeatletter
\newcommand{\imod}[1]{\allowbreak\mkern4mu({\operator@font mod}\,\,#1)}
\makeatother

\theoremstyle{plain}
\newtheorem{theorem}{Theorem} 
\newtheorem{conj}[theorem]{Conjecture} 

\newtheorem{thm}{Theorem}[section] 

\newtheorem{prop}[thm]{Proposition}

\newtheorem*{theorem*}{Theorem} 
\newtheorem*{conj*}{Conjecture}

\theoremstyle{definition}
\newtheorem{rem}[thm]{Remark}
\newtheorem{defn}[thm]{Definition}
\newtheorem*{deff}{Definition}

\newtheorem{remk}[theorem]{Remark}

\begin{document}

\title[On the depth of subgroups of simple groups]{On the depth of subgroups of 
simple groups}

\author{Timothy C. Burness}
\address{T.C. Burness, School of Mathematics, University of Bristol, Bristol BS8 1UG, UK}
\email{t.burness@bristol.ac.uk}

\date{\today} 

\begin{abstract}
The depth of a subgroup $H$ of a finite group $G$ is a positive integer defined with respect to  the inclusion of the corresponding complex group algebras $\mathbb{C}H \subseteq \mathbb{C}G$. This notion was originally introduced by Boltje, Danz and K\"{u}lshammer in 2011, and it has been the subject of numerous papers in recent years. In this paper, we study the depth of core-free subgroups, which allows us to apply powerful computational and probabilistic techniques that were originally designed for studying bases for permutation groups. We use these methods to prove a wide range of new results on the depth of subgroups of almost simple groups, significantly extending the scope of earlier work in this direction. For example, we establish best possible bounds on the depth of irreducible subgroups of classical groups and primitive subgroups of symmetric groups. And with the exception of a handful of open cases involving the Baby Monster, we calculate the exact depth of every subgroup of every almost simple sporadic group. We also present a number of open problems and conjectures.
\end{abstract}

\maketitle

\section{Introduction}\label{s:intro}

Let $G$ be a finite group and let $H$ be a proper subgroup. We begin by recalling the definition of the \emph{depth} of $H$ in $G$, which was originally introduced by Boltje, Danz and K\"{u}lshammer in \cite{BDK}. This positive integer is defined in terms of the inclusion of complex group algebras $\mathbb{C}H \subseteq \mathbb{C}G$, extending earlier notions of depth that had been defined and studied for other families of algebras, such as von Neumann algebras, Hopf algebras and Frobenius algebras (see \cite{K,KK,KN,KS}, for example).

Given a positive integer $n$, we say that the inclusion $\mathbb{C}H \subseteq \mathbb{C}G$ has depth $2n$ if $(\mathbb{C}G)^{n+1}$ is isomorphic, as a $(\mathbb{C}G,\mathbb{C}H)$-bimodule, to a direct summand of $\bigoplus_{i=1}^k (\mathbb{C}G)^n$ for some $k \geqs 1$, where
\[
(\mathbb{C}G)^m := \mathbb{C}G \otimes_{\mathbb{C}H} \cdots \otimes_{\mathbb{C}H} \mathbb{C}G \;\; \mbox{(with $m$ factors).}
\]
And the inclusion $\mathbb{C}H \subseteq \mathbb{C}G$ has depth $2n+1$ if the same conclusion holds with respect to $(\mathbb{C}H,\mathbb{C}H)$-bimodules. Finally, $\mathbb{C}H$ has depth $1$ in $\mathbb{C}G$ if there is a positive integer $k$ such that $\mathbb{C}G$ is isomorphic to a direct summand of $\bigoplus_{i=1}^k \mathbb{C}H$ as a $(\mathbb{C}H,\mathbb{C}H)$-bimodule. 

By \cite[Theorem 4.1]{BDK}, every group algebra inclusion $\mathbb{C}H \subseteq \mathbb{C}G$ has depth $n$ for some positive integer $n$ (and in view of \cite[Corollary 6.11]{BKK}, one can  take $n \leqs 2|G:N_G(H)|$). And from the definition, we see that if $\mathbb{C}H \subseteq \mathbb{C}G$ has depth $n$, then it also has depth $n+1$, so it is natural to consider the minimal depth of the inclusion. This leads us to the following definition.

\begin{deff}
The \emph{depth} of $H$ in $G$, denoted $d_G(H)$, is defined to be the minimal depth of the inclusion of complex group algebras $\mathbb{C} H \subseteq \mathbb{C}G$.
\end{deff}

In earlier work, it has been proved that $d_G(H) \leqs 2$ if and only if $H$ is a normal subgroup of $G$ (see \cite[Corollary 3.2]{KK}). Moreover, a theorem of Boltje and K\"{u}lshammer \cite[Theorem 1.7]{BK} reveals that $d_G(H) = 1$ if and only if $G= HC_G(x)$ for all $x \in H$ (in particular, the depth of the trivial subgroup, and also $G$ itself, is $1$). Given any positive integer $n$, it is known that there exists a finite group $G$ with a subgroup $H$ such that $d_G(H) = n$. For example, $d_{S_{m+1}}(S_{m}) = 2m-1$ by \cite[Theorem 6.19]{BKK}, while the problem for even depth was resolved more recently in \cite{JBH}. It is also worth noting that the depth of $H$ in $G$ does not change if $\mathbb{C}$ is replaced by any field of characteristic zero (see \cite[Remark 4.5(a)]{BDK}). 

There are several different methods for calculating or bounding $d_G(H)$, and we briefly recall some of the main approaches in Section \ref{ss:depth}. In particular, an interesting connection with bases for permutation groups arises in the special case where $H$ is a \emph{core-free} subgroup of $G$ (recall that $H$ is core-free if it does not contain a nontrivial normal subgroup of $G$).

In order to describe this connection, let us assume $H$ is nontrivial and core-free,  in which case $d_G(H) \geqs 3$ by the aforementioned result of Kadison and K\"{u}lshammer \cite{KK}. We may view $G \leqs {\rm Sym}(\O)$ as a transitive permutation group with point stabilizer $H$ and we recall that a \emph{base} for $G$ is a subset of $\O$ with trivial pointwise stabilizer. The \emph{base size} of $G$, denoted $b(G,H)$, is the minimal size of a base. Equivalently, $b(G,H)$ is the minimal integer $n$ such that there exist $n$ conjugates of $H$ with trivial intersection. This fundamental invariant has been extensively studied in the permutation group theory literature for many decades, with bases finding a wide range of applications and connections to other areas (we refer the reader to the survey articles \cite{BC} and \cite[Section 5]{B180} for further details and references).

The main result connecting the base size of $G$ with the depth of a core-free subgroup $H$ is   \cite[Theorem 6.9]{BKK}, which implies that 
\begin{equation}\label{e:base}
d_G(H) \leqs 2b(G,H)-1.
\end{equation}
In particular, if $b(G,H) = 2$ then $d_G(H) = 3$. It is also worth noting that there are examples with $d_G(H)<b(G,H)$ (see Remark \ref{r:q51} for instance, which answers  a question from \cite{JHH}). And if $K$ is a subgroup of a core-free subgroup $H$ of $G$, then we have $d_G(K) \leqs d_G(H)$ (see Proposition \ref{p:graph}(vi)).

In this paper, we focus on the depth of subgroups of almost simple groups, where we recall that a finite group $G$ is \emph{almost simple} if the socle of $G$, denoted $G_0$, is a nonabelian simple group. In particular, by identifying $G_0$ with its group of inner automorphisms, we have 
\[
G_0 \normeq G \leqs {\rm Aut}(G_0).
\]
For example, the groups $S_n$ and ${\rm PGL}_2(q)$ are almost simple for $n \geqs 5$ and $q \geqs 4$, respectively.

Determining the exact depth of a given subgroup is a difficult problem, in general, but several special cases in the almost simple setting have been studied in earlier work. For example,  \cite[Theorem 6.19]{BKK} shows that the depth of $H = S_{n-1}$ in $G = S_n$ is $2n-3$ (notice that this coincides with $2b(G,H) -1$, demonstrating the sharpness of the upper bound in \eqref{e:base}). In \cite{FKR}, this was extended to the maximal intransitive subgroups of $G = S_n$, which are of the form $H = S_k \times S_{n-k}$ with $k \leqs n/2$. For example, if $n \geqs (k^2+2)/2$ then the main theorem of \cite{FKR} gives 
\[
d_G(H) = 2\left\lceil \frac{2(n-1)}{k+1}\right\rceil -1,
\]
with upper and lower bounds on $d_G(H)$ presented in the remaining cases. Depth for subgroups of certain rank one simple groups of Lie type has also been studied, namely ${\rm PSL}_2(q)$ in \cite{F1,F2}, and the Suzuki and small Ree groups in \cite{HHP1,HHP2,JHH}. We refer the reader to Table \ref{tab:as} in Section \ref{ss:as} for a convenient summary of existing results for almost simple groups.

Bases for finite primitive groups have been intensively studied over the last 25 years and there continues to be a lot of interest and activity in this area. In particular, there is an extensive literature on bases for primitive actions of almost simple groups, which we can use (via \eqref{e:base}) to shed new light on the depth of subgroups of such groups. Indeed, our main goal in this paper is to systematically translate and apply results and techniques that have proved highly effective in the study of bases in order to significantly extend and enhance our understanding of the depth of subgroups of almost simple groups. And along the way, we will highlight several open problems, which will be a focus of future work. 

In order to state our first result, we need to recall some terminology from permutation group theory. Let $G \leqs {\rm Sym}(\O)$ be an almost simple primitive group with socle $G_0$ and point stabilizer $H$. Roughly speaking, we say that $G$ is \emph{standard} if $G_0 = A_n$ is an alternating group and $H \cap G_0$ is intransitive or imprimitive on $\{1, \ldots, n\}$, or if $G_0$ is a classical group and $H \cap G_0$ acts reducibly on the natural module for $G_0$ (see \cite[Definition 1.1]{B07} for the formal definition). Otherwise, $G$ is said to be \emph{non-standard}. Typically, the base size of a standard group can be arbitrarily large. For example, if $G = {\rm PGL}_n(q)$ with $q \geqs 3$ and $\O$ is the set of $1$-dimensional subspaces of the natural module, then $b(G,H) = n+1$. However, this behaviour is in stark contrast to a remarkable conjecture of Cameron from the late 1990s (see \cite[p.122]{Cam}), which asserts that all non-standard groups admit small bases. More precisely, Cameron conjectured that $b(G,H) \leqs 7$ for every non-standard group $G$, with equality if and only if $G$ is the Mathieu group ${\rm M}_{24}$ in its natural action on $24$ points (in which case $H = {\rm M}_{23}$). A version of Cameron's conjecture with an undetermined absolute constant was proved by Liebeck and Shalev \cite{LSh99}, and the precise form of the conjecture was established in later work by Burness et al. \cite{B07,BGS,BLS,BOW}. In both \cite{LSh99} and the subsequent work, a key ingredient is a powerful probabilistic approach for bounding $b(G,H)$, which is based on fixed point ratio estimates (we will discuss this in Section \ref{sss:bases}).

These results on bases have immediate applications to the depth problem, as demonstrated by our first main result.

\begin{theorem}\label{t:ns}
Let $G \leqs {\rm Sym}(\O)$ be a non-standard almost simple primitive permutation group with socle $G_0$ and point stabilizer $H$.
\begin{itemize}\addtolength{\itemsep}{0.2\baselineskip}
\item[{\rm (i)}] If $G_0 = A_n$ is an alternating group, then $d_G(H) \leqs 9$, with equality if and only if $G = S_6$ and $H = S_5$ is primitive. 
\item[{\rm (ii)}] If $G_0$ is a sporadic group, then $d_G(H) \leqs 11$, with equality if and only if $G = {\rm M}_{24}$ and $H = {\rm M}_{23}$.
\item[{\rm (iii)}] If $G_0$ is classical, then $d_G(H) \leqs 7$, with equality if $G = {\rm U}_6(2).2$ and $H = {\rm U}_4(3).2^2$.
\item[{\rm (iv)}] If $G_0$ is an exceptional group of Lie type, then $d_G(H) \leqs 11$.
\end{itemize}
\end{theorem}

\begin{remk}\label{r:ns}
Let us record some comments on the statement of Theorem \ref{t:ns}.
\begin{itemize}\addtolength{\itemsep}{0.2\baselineskip}
\item[{\rm (a)}] For $G_0 = A_n$, the non-standard hypothesis simply means that $H \cap G_0$ acts primitively on $\{1,\ldots, n\}$, so part (i) reveals that all primitive subgroups of symmetric and alternating groups have small depth. We refer the reader to Theorem \ref{t:sn} below for a complete classification of the primitive subgroups with depth greater than $3$.  

\item[{\rm (b)}] The upper bound in part (ii) is a special case of Theorems \ref{t:main0} and \ref{t:main2} below, where we give the exact depth of every maximal subgroup of every almost simple sporadic group. In fact, with the exception of a handful of open cases involving the Baby Monster, we are able to calculate the exact depth of every subgroup of every almost simple sporadic group (see Theorems \ref{t:main0}, \ref{t:main} and \ref{t:main2}).

\item[{\rm (c)}] Part (iii) shows that irreducible subgroups of simple classical groups have small depth. Indeed, by putting aside a small number of cases that are excluded by our definition of a non-standard group, such as ${\rm O}_{n}^{\pm}(q) < {\rm Sp}_n(q)$ with $q$ even, we deduce that every irreducible subgroup has depth at most $7$. Note that this upper bound is best possible even when $G$ is simple. For example, if $G = {\rm U}_4(2)$ and $H = 3^3.S_4$ is an irreducible subgroup of type ${\rm GU}_1(2) \wr S_4$, then one can check that $d_G(H) = 7$. It remains an open problem to determine all the non-standard classical groups with $d_G(H) = 7$.

\item[{\rm (d)}] The bound in (iv) follows immediately from the main theorem of \cite{BLS}, which gives $b(G,H) \leqs 6$ for every core-free subgroup $H$ of an almost simple exceptional group $G$. By \cite[Theorem 1]{B18}, we have $b(G,H) = 6$ if and only if $(G_0,H) = (E_7(q),P_7)$, $(E_6(q),P_1)$ or $(E_6(q),P_6)$, where $P_m$ denotes a maximal parabolic subgroup of $G$, labelled in the usual way. So away from these cases, we get $d_G(H) \leqs 9$ and we know that there are examples with $d_G(H) = 7$. For instance, if $(G,H) = (E_6(2),P_1)$ then by working with the  permutation representation of $G$ on $G/H$ of degree $139503$ we can construct the character tables of $G$ and $H$, which allows us to compute $d_G(H) = 7$ using the character-theoretic method described in Section \ref{sss:inc}. But it remains an open problem to determine whether or not there exists a subgroup $H$ of an almost simple exceptional group $G$ with $d_G(H) \geqs 8$. 
\end{itemize}
\end{remk}

Next we present more refined versions of Theorem \ref{t:ns} for almost simple groups with socle an alternating group or sporadic group. Note that Table \ref{tab:sn} is presented in Section \ref{s:tab} at the end of the paper, where the possibilities for $H$ are recorded up to conjugacy in $G$.

\begin{theorem}\label{t:sn}
Let $G$ be an almost simple group with socle $G_0 = A_n$ and let $H$ be a core-free subgroup of $G$ such that $H \cap G_0$ acts primitively on $\{1, \ldots, n\}$. Set $d=d_G(H)$. Then either 
\begin{itemize}\addtolength{\itemsep}{0.2\baselineskip}
\item[{\rm (i)}] $d = 3$; or
\item[{\rm (ii)}] $4 \leqs d \leqs 9$ and $(G,H,d)$ is one of the cases listed in Table \ref{tab:sn}.
\end{itemize}
In particular, $d \leqs 9$, with equality if and only if $G = S_6$ and $H=S_5$.  
\end{theorem}

\begin{remk}\label{r:sn}
Some remarks on the statement of Theorem \ref{t:sn}.
\begin{itemize}\addtolength{\itemsep}{0.2\baselineskip}
\item[{\rm (a)}] It is worth noting that $d_G(H)$ is even if and only if $G = S_{10}$ and $H = A_6.2^2$, in which case $d_G(H) = 4$.

\item[{\rm (b)}] Suppose $G = S_n$ and $H = S_k \times S_{n-k}$ is a maximal intransitive subgroup of $G$ with $k \leqs n/2$. Here we can identify the action of $G$ on $G/H$ with the natural action of $G$ on the set of $k$-element subsets of $\{1, \ldots, n\}$. Building on earlier work of Halasi \cite{Halasi}, the exact base size $b(G,H)$ has recently been determined in all cases (see \cite{RV,MS0}). And as mentioned earlier, several results on $d_G(H)$ in this setting are presented in \cite{FKR}, but the exact depth of $H$ is not known for all $n$ and $k$.

\item[{\rm (c)}] Suppose $G = S_n$ and $H = S_a \wr S_b$ is imprimitive, where $n = ab \geqs 6$ and $a,b \geqs 2$. Here the exact base size $b(G,H)$ is computed in \cite{BGL2,JJ,MS}. For example, if $a \leqs b$ then \cite[Theorem 2]{BGL2} gives $b(G,H) \leqs 4$, with equality if and only if $(a,b) = (2,3)$, and as a consequence we deduce that $d_G(H) \leqs 5+\epsilon$, where $\e = 2$ if $(a,b) = (2,3)$, otherwise $\e = 0$ (it is easy to check that $d_{S_6}(S_2 \wr S_3) = 7$). Moreover, \cite[Theorem 1.2]{JJ} gives $b(G,H) = 2$ if and only if $a \geqs 3$ and $b \geqs \max\{a+3,8\}$, which implies that $d_{S_{24}}(S_3 \wr S_8) = 3$, for example. In the general case, the main theorem of \cite{MS} implies that $b(G,H) \leqs \log_2n+2$ and thus \eqref{e:base} yields $d_G(H) \leqs 2\log_2n+3$. 
\end{itemize}
\end{remk}

Next assume $G = S_n$ or $A_n$ and let $H$ be an arbitrary core-free subgroup of $G$. Then \cite[Theorem 1(i)]{AB} gives $b(G,H) \leqs n - |S_n:G|$, with equality if $H = S_{n-1} \cap G$. So by combining the bound on $d_G(H)$ in \eqref{e:base} with two results from 
\cite{BKK} (namely, Theorem 6.19 and Proposition A.5), we can prove the following.

\begin{theorem}\label{t:sn2}
Let $G = S_n$ or $A_n$ with $n \geqs 3$. Then
\[
d_G(H) \leqs \left\{\begin{array}{ll}
2n-3 & \mbox{if $G = S_n$} \\
2n-2\lceil \sqrt{n-1}\rceil -1 & \mbox{if $G = A_n$}
\end{array}\right.
\]
for every subgroup $H$ of $G$, with equality if and only if one of the following holds:
\begin{itemize}\addtolength{\itemsep}{0.2\baselineskip}
\item[{\rm (i)}] $H = S_{n-1} \cap G$;
\item[{\rm (ii)}] $(G,H) = (A_6, (S_3 \wr S_2) \cap G)$, $(A_6, (S_3 \times S_3) \cap G)$, $(A_5,D_{10})$, $(A_4,C_2)$ or $(A_3,A_3)$.
\end{itemize}
\end{theorem}

Now suppose $G$ is an almost simple classical group, with $H$ a core-free subgroup. By combining the upper bound on $d_G(H)$ in part (iii) of Theorem \ref{t:ns} with upper bounds on $b(G,H)$ derived in \cite[Section 3]{HLM} for subspace actions, we obtain the following result (here $r$ denotes the untwisted Lie rank of $G$, which coincides with the rank of the ambient simple algebraic group).

\begin{theorem}\label{t:class}
There exist absolute constants $a$ and $b$ such that 
\[
d_G(H) \leqs ar+b
\]
for every core-free subgroup $H$ of every almost simple classical group $G$ of rank $r$. Moreover, we can take $a=4$ and $b=29$.
\end{theorem}

\begin{remk}\label{r:class}
As noted above, the base size bounds in \cite[Section 3]{HLM} for subspace actions of classical groups are a key ingredient in the proof of Theorem \ref{t:class}. And we expect that a linear upper bound in terms of the rank is best possible. 

\begin{itemize}\addtolength{\itemsep}{0.2\baselineskip}
\item[{\rm (a)}] For example, consider the simple group $G = {\rm GL}_n(2)$ with $n \geqs 3$ and let $V$ be the natural module for $G$. As noted in the proof of \cite[Proposition 2.13]{AB}, we have $b(G,H) \leqs n = r+1$ for every proper subgroup $H$ of $G$, with equality if and only if $H$ is the stabilizer of a $1$-space or a hyperplane in $V$ (in which case, $H$ is a maximal parabolic subgroup of type $P_1$ or $P_{n-1}$). Then \eqref{e:base} gives $d_G(H) \leqs 2r+1$ for every subgroup $H$ and it appears that equality holds when $H = P_1$ or $P_{n-1}$ (we have checked this computationally for $3 \leqs n \leqs 9$). 

\item[(b)] Referring to the above example, it is interesting to consider how the depth of $H$ in $G$ relates to the depth of $W_H$ in $W$, where $W = S_n$ is the Weyl group of $G$ and $W_H = S_{n-1}$ is the corresponding parabolic subgroup of $W$. By the aforementioned theorem of Burciu et al. \cite{BKK} we have $d_{W}(W_H) = 2r-1$ and so for $n \leqs 9$ we deduce that $d_G(H) = d_{W}(W_H)+2$.
\end{itemize}
\end{remk}

The proof of Cameron's base size conjecture for almost simple sporadic groups is presented in \cite{BOW}. In fact, the exact base size of every primitive almost simple sporadic group is determined in \cite{BOW}, with the exception of two specific cases involving the Baby Monster, which were later resolved in \cite{NNOW} using more sophisticated computational methods. It remains a difficult open problem to determine $b(G,H)$ precisely for every core-free subgroup $H$ of every  almost simple sporadic group $G$. So it is natural to ask if it is possible to compute $d_G(H)$ at the same level of generality.

Here we present three main results on the depth of subgroups of almost simple sporadic groups, which bring us very close to achieving this goal. We begin in Theorem \ref{t:main0} by determining the exact depth of every maximal subgroup of every simple sporadic group. 
Note that Table \ref{tab:max} is presented in  Section \ref{s:tab} (also see Remark \ref{r:tabs}).

\begin{theorem}\label{t:main0}
Let $G$ be a simple sporadic group, let $H$ be a maximal subgroup of $G$ and set $d = d_G(H)$. Then either $d = 3$, or $(G,H,d)$ is recorded in Table \ref{tab:max}. In particular, $d \leqs 11$, with equality if and only if $G = {\rm M}_{24}$ and $H = {\rm M}_{23}$.
\end{theorem}

In order to extend Theorem \ref{t:main0} to arbitrary subgroups, it will be convenient to introduce some additional terminology. Let us say that a proper subgroup $H$ of $G$ is \emph{$k$-maximal} if $k$ is the smallest positive integer such that there exists a chain of subgroups
\[
H = H_k < H_{k-1} < \cdots < H_0 = G,
\]
where each $H_i$ is maximal in $H_{i-1}$ (so a subgroup is $1$-maximal if and only if it is maximal). Let $\mathcal{M}_k$ be the set of $k$-maximal subgroups of $G$. Then a slightly simplified version of our main result on the depth of subgroups of simple sporadic groups is Theorem \ref{t:main} below. Once again, Tables \ref{tab:main2} and \ref{tab:main3} are presented in Section \ref{s:tab} at the end of the paper (see Remark \ref{r:tabs}), noting that the subgroups $H$ arising in parts (ii) and (iii) are recorded up to conjugacy in $G$.

\begin{theorem}\label{t:main}
Let $G$ be a simple sporadic group and set $d = d_G(H)$, where $H \in \mathcal{M}_k$ is nontrivial and $k \geqs 2$. Then one of the following holds:
\begin{itemize}\addtolength{\itemsep}{0.2\baselineskip}
\item[{\rm (i)}] $d=3$;
\item[{\rm (ii)}] $H \in \mathcal{M}_2$ and $(G,H,d)$ is one of the cases in Table \ref{tab:main2}.
\item[{\rm (iii)}] $H \in \mathcal{M}_3$, $d=5$ and $(G,H)$ is one of the cases in Table \ref{tab:main3}.
\item[{\rm (iv)}] $G = \mathbb{B}$, $H < A$ and $d \leqs 5$, where $A = 2.{}^2E_6(2).2$ or $2^{1+22}.{\rm Co}_2$.
\end{itemize}
\end{theorem}

\begin{remk}\label{r:main1}
Let us record a couple of comments on the statement of Theorem \ref{t:main}.
\begin{itemize}\addtolength{\itemsep}{0.2\baselineskip}
\item[{\rm (a)}] Notice that Theorem \ref{t:main} implies that if $G \ne \mathbb{B}$ is a simple sporadic group, then every nontrivial $k$-maximal subgroup of $G$ with $k \geqs 4$  has depth $3$. We do not know if the Baby Monster $G$ has a $4$-maximal subgroup $H$ with $d_G(H) >3$, but our analysis does imply that such a subgroup would need to be part of a chain
\[
\mbox{ } \hspace{10mm} H < 2.J< 2.{}^2E_6(2).2<G,
\]
where $J$ is a maximal parabolic subgroup of ${}^2E_6(2).2$.

\item[{\rm (b)}] Several computational difficulties arise when $G$ is the Baby Monster. For example, $G$ does not admit a suitable permutation or matrix representation, which would allow us to proceed as in other cases. In particular, the character tables of many $2$-maximal subgroups of $G$ are not known and this is the main reason why it remains an open problem to determine the exact depth of every subgroup of $G$. As commented above, similar issues arose in \cite{BOW}, where the exact base size $b(G,H)$ is computed for every maximal subgroup $H$, with two exceptions that were dealt with separately in a later paper \cite{NNOW}, using specialized computational techniques. As indicated in part (iv) of Theorem \ref{t:main}, for $G = \mathbb{B}$ we have reduced the problem to the situation where $H$ is a proper subgroup of a maximal subgroup $A$ of $G$, with $A = 2.{}^2E_6(2).2$ or $2^{1+22}.{\rm Co}_2$. We refer the reader to Remarks \ref{r:2e6} and \ref{r:co2} for further reductions and a detailed discussion of these remaining open cases.
\end{itemize}
\end{remk}

Our final result for sporadic groups is the following, which gives the depth of every subgroup of every almost simple sporadic group of the form $G = G_0.2$. (Again, the tables referred to in part (iii) are presented in Section \ref{s:tab}.)

\begin{theorem}\label{t:main2}
Let $G = G_0.2$ be an almost simple sporadic group and set $d = d_G(H)$, where $H$ is a proper nontrivial subgroup of $G$. Then one of the following holds:
\begin{itemize}\addtolength{\itemsep}{0.2\baselineskip}
\item[{\rm (i)}] $d=3$;
\item[{\rm (ii)}] $H = G_0$ and $d = 2$;
\item[{\rm (iii)}] $H \in \mathcal{M}_k$ with $k \leqs 4$ and $(G,H,d)$ is one of the cases in Table \ref{tab:main41} or \ref{tab:main42}.
\end{itemize}
\end{theorem}

\renewcommand{\arraystretch}{1.3}

\begin{remk}\label{r:t7}
Let $G \ne \mathbb{B}$ be an almost simple sporadic group and let $H \in \mathcal{M}_k$ be a nontrivial $k$-maximal subgroup of $G$. By combining Theorems \ref{t:main0}-\ref{t:main2}, it follows that $d_G(H) = 3$ if $k \geqs 5$. In addition, if $H \in \mathcal{M}_4$ then $d_G(H) \geqs 4$ if and only if one of the following holds:
\begin{itemize}\addtolength{\itemsep}{0.2\baselineskip}
\item[{\rm (a)}] $G = {\rm Fi}_{22}.2$ and $H = {\rm U}_4(3).2_2$ or $\O_8^{+}(2)$ (in the former  case, $H$ contains an involutory graph automorphism $x$ with $C_S(x) = {\rm PSp}_4(3)$, where $S = {\rm U}_4(3)$ is the socle of $H$). 

\item[{\rm (b)}] $G = {\rm Fi}_{24}'.2$ and $H = {\rm P\O}_8^{+}(3).3$.
\end{itemize}
In each of these cases we have $d_G(H)=5$ and we can embed $H$ in $G$ via the following maximal chains:
\[
\begin{array}{l}
{\rm U}_4(3).2_2 < 2 \times {\rm U}_4(3).2_2 < 2 \times {\rm U}_4(3).2^2 < S_3 \times {\rm U}_4(3).2^2 < {\rm Fi}_{22}.2 \\
\O_8^{+}(2) < \O_8^{+}(2) \times 2 < \O_8^{+}(2).2 \times 2 < \O_8^{+}(2).S_3 \times 2 < {\rm Fi}_{22}.2 \\
{\rm P\O}_8^{+}(3).3 < 2 \times {\rm P\O}_8^{+}(3).3 < 2 \times {\rm P\O}_8^{+}(3).S_3 < S_3 \times {\rm P\O}_8^{+}(3).S_3 < {\rm Fi}_{24}'.2
\end{array}
\]
\end{remk} 

\renewcommand{\arraystretch}{1}

\begin{remk}\label{r:d4}
Theorems \ref{t:main0}-\ref{t:main2} show that very few core-free subgroups of almost simple sporadic groups have even depth, which resonates with Remark \ref{r:sn}(a) on primitive subgroups of $S_n$ and $A_n$. For example, if $G$ is a simple sporadic group and $H$ is maximal, then $d_G(H)$ is even if and only if $(G,H) = ({\rm HS}, {\rm L}_3(4).2)$ or $({\rm Fi}_{24}', 3^{1+10}.{\rm U}_5(2).2)$, with $d_G(H) = 4$ in both cases. Similarly, if $G \ne \mathbb{B}$ is simple and $H$ is non-maximal, then there are just three cases with $d_G(H)$ even. But let us observe that there are infinitely many pairs $(G,H)$, where $G$ is simple and $d_G(H)$ is even. Indeed, if we take $n = d^2+1$, where $d \geqs 3$ is odd, then \cite[Proposition A.5]{BKK} gives $d_G(H) = 2d^2-d+1$ for $G = A_n$ and $H = A_{n-1}$. This leads us naturally to the following question: Does every even integer $n \geqs 4$ arise as the depth of a subgroup of an almost simple group?
\end{remk}

In \cite{BG23}, Burness and Guralnick introduced the concept of a \emph{strongly base-two} group. This is a finite group $G$ with $\a(G) \leqs 1$, where $\a(G)$ is the number of conjugacy classes of core-free subgroups $H$ with $b(G,H) \geqs 3$. Now if $G$ is a (nonabelian) finite simple group, then $\a(G) \geqs 1$ since $G$ has a proper subgroup $H$ with $|H|^2 \geqs |G|$ (see \cite[Theorem 2.1]{BG23}). Moreover, for $G$ simple, \cite[Theorem 2]{BG23} tells us that $\a(G) = 1$ if and only if $G$ is one of the following:
\[
{\rm L}_2(q) \mbox{ with $q \geqs 23$ odd and $(q-1)/2$ prime},\; {}^2B_2(q) \mbox{ with $q-1$ prime}, \; {\rm J}_1, \; \mathbb{M},
\]
where $\mathbb{M}$ is the Monster sporadic group. It is natural to consider the analogous problem with respect to depth, so with this in mind we define $\delta(G)$ to be the number of conjugacy classes of core-free subgroups $H$ of $G$ with $d_G(H) \geqs 4$. 

If $G$ is a simple sporadic group, then by inspecting Table \ref{tab:max} we deduce that $\delta(G) = 0$ if and only if $G = {\rm J}_1$, ${\rm J}_3$, ${\rm J}_4$ or ${\rm Th}$ (see Remark \ref{r:delta}(a) below). Our next result reveals that these are the only simple groups with $\delta(G) = 0$.

\begin{theorem}\label{t:delta}
Let $G$ be a finite simple group. Then $d_G(H) = 3$ for every proper nontrivial subgroup $H$ of $G$ if and only if $G = {\rm J}_1$, ${\rm J}_3$, ${\rm J}_4$ or ${\rm Th}$.
\end{theorem}

\begin{remk}\label{r:delta}
Let $G$ be a simple group and let $H$ be a proper nontrivial subgroup. 

\begin{itemize}\addtolength{\itemsep}{0.2\baselineskip}
\item[{\rm (a)}] Since $d_G(K) \leqs d_G(H)$ for every subgroup $K$ of $H$ (see Proposition \ref{p:graph}(vi)), it follows that $\delta(G) = 0$ if and only if $d_G(H) = 3$ for every maximal subgroup $H$ of $G$.

\item[{\rm (b)}] In Section \ref{sss:kuls} we define a graph $\mathcal{K}(G,H)$, where the vertices are labelled by the irreducible (complex) characters of $H$, with $\a$ and $\b$ adjacent if and only if the induced characters $\a^G$ and $\b^G$ have a common irreducible constituent. Then $d_G(H) = 3$ if and only if the graph $\mathcal{K}(G,H)$ is complete (see Proposition \ref{p:graph}(iv)), so we can view Theorem \ref{t:delta} as a classification of the finite simple groups $G$ with the property that $\mathcal{K}(G,H)$ is complete for every proper nontrivial subgroup $H$ of $G$.
\end{itemize}
\end{remk}
 
Next we turn to the depth of nilpotent and solvable subgroups of almost simple groups. Our main result for nilpotent subgroups is the following. Here parts (i) and (ii) follow from earlier work of Zenkov \cite{Zenkov14,Zen20}, while part (iii) is an immediate consequence of a recent theorem of Burness and Huang \cite[Corollary B]{BH1}, which in turn resolves a conjecture of Vdovin from 2002 (see Problem 15.40 in the \emph{Kourovka Notebook} \cite{Kou}).

\begin{theorem}\label{t:nilp}
Let $G$ be an almost simple group with socle $G_0$ and let $H$ be a nontrivial nilpotent subgroup of $G$. 
\begin{itemize}\addtolength{\itemsep}{0.2\baselineskip}
\item[{\rm (i)}] We have $d_G(H) \leqs 5$, with equality if $G = S_8$ and $H$ is a Sylow $2$-subgroup of $G$.
\item[{\rm (ii)}] If $G_0$ is an alternating or sporadic group, then
\[
d_G(H) = \left\{\begin{array}{ll}
5 & \mbox{if $G = S_8$ and $H$ is a Sylow $2$-subgroup} \\
3 & \mbox{otherwise.}
\end{array}\right.
\]
\item[{\rm (iii)}] If $G = G_0$ is simple, then $d_G(H) = 3$.
\end{itemize}
\end{theorem}

\begin{remk}\label{r:nilp}
We record a couple of remarks on the statement of Theorem \ref{t:nilp}.
\begin{itemize}\addtolength{\itemsep}{0.2\baselineskip}
\item[{\rm (a)}] The upper bound in part (i) of Theorem \ref{t:nilp} is a special case of a more general result due to Zenkov. Indeed, the main theorem of \cite{Zen21} implies that if $G$ is a finite group with trivial Fitting subgroup and $H$ is a nilpotent subgroup of $G$, then $b(G,H) \leqs 3$ and thus $d_G(H) \leqs 5$ via \eqref{e:base}.

\item[{\rm (b)}] It remains an open problem to determine $b(G,H)$ for every nilpotent subgroup $H$ of every almost simple group $G$ of Lie type. There are partial results for low rank groups, and there is a complete answer in the special case where $H$ is a Sylow $p$-subgroup (see the main theorem of \cite{BH2}, which extends earlier work of Mazurov and Zenkov \cite{Z_92p,ZM} from the 1990s). And if $G$ is simple, then \cite[Corollary B]{BH1} states that $b(G,H) = 2$ for every nontrivial nilpotent subgroup $H$.

\item[{\rm (c)}] The main theorem of \cite{ZM} states that if $G$ is simple and $H$ is a Sylow $p$-subgroup, where $p$ is a prime divisor of $|G|$, then $b(G,H) = 2$ and thus $d_G(H) = 3$. As originally noted by Zenkov \cite{Z_92p}, this does not extend to almost simple groups and a complete classification of the pairs $(G,H)$ with $b(G,H) = 3$ has very recently been obtained in \cite{BH2}. For example, if $G = {\rm PGL}_2(q)$ and $H = D_{2(q+1)}$, where $q$ is a Mersenne prime, then $H$ is a Sylow $2$-subgroup and $b(G,H) = 3$, which means that $d_G(H) \leqs 5$. In this case, we can  take $\a,\b \in {\rm Irr}(H)$, where $\a$ is trivial and $\b$ is the nontrivial linear character of $H$ inflated from 
$H/N$, where $N$ is the cyclic subgroup of order $q+1$. Then it is straightforward to show that $\la \a^G , \b^G \ra = 0$, where $\a^G$ and $\b^G$ are the induced characters and $\la \; , \, \ra$ is the usual inner product for characters of $G$. Therefore, Proposition \ref{p:graph}(v) implies that $d_G(H) = 5$. It remains an open problem to determine the exact depth of every Sylow subgroup of every almost simple group (see \cite[Remark 7]{BH2} for further discussion).
\end{itemize}
\end{remk}

Finally, our main result for solvable subgroups is the following (here $P_2$ denotes the stabilizer in $G$ of a $2$-dimensional subspace of the natural module for ${\rm L}_4(3)$).

\begin{theorem}\label{t:solv}
Let $G$ be an almost simple group and let $H$ be a solvable maximal subgroup of $G$. Then $d_G(H) \leqs 9$, with equality if and only if $G = {\rm L}_4(3).2_2$ or ${\rm L}_4(3).2^2$ and $H = P_2$.
\end{theorem}

\begin{remk}\label{r:solv}
Some comments on the statement of Theorem \ref{t:solv} are in order.
\begin{itemize}\addtolength{\itemsep}{0.2\baselineskip}
\item[{\rm (a)}] The bound $d_G(H) \leqs 9$ is an immediate consequence of the main result of \cite{B21}, which gives $b(G,H) \leqs 5$ for all $G$ and $H$ as in Theorem \ref{t:solv}. By carefully inspecting the cases with $b(G,H) = 5$, it is straightforward to show that either $d_G(H) \leqs 7$ or one of the following holds:

\vspace{2mm}

\begin{itemize}\addtolength{\itemsep}{0.2\baselineskip}
\item $G = {\rm U}_5(2)$, $H = P_1$ is the stabilizer of a $1$-dimensional totally isotropic subspace of the natural module for $G$ and we compute $d_G(H) = 8$; or

\item $G = {\rm L}_4(3).2_2$ or ${\rm Aut}({\rm L}_4(3)) = {\rm L}_4(3).2^2$, $H = P_2$ and $d_G(H) = 9$. Here ${\rm L}_4(3).2_2$ can be identified among the almost simple groups of the form ${\rm L}_4(3).2$ by the fact that it contains an involutory graph automorphism $x$ with $C_{G_0}(x) = {\rm PGSp}_4(3)$, where $G_0 = {\rm L}_4(3)$.
\end{itemize}

\vspace{1mm}

We refer the reader to Theorem \ref{t:solv2} for a more refined statement concerning the almost simple groups $G$ with a solvable maximal subgroup $H$ such that $d_G(H) > 5$.

\item[{\rm (b)}] If $G$ is an almost simple sporadic group and $H$ is \emph{any} solvable subgroup of $G$, then the main theorem of \cite{B23} gives $b(G,H) \leqs 3$ and thus $d_G(H) \leqs 5$. Note that equality is possible. For example, by inspecting Table \ref{tab:max} we see that $d_G(H) = 5$ when $G = {\rm M}_{11}$ and $H = {\rm U}_3(2){:}2$ is a solvable maximal subgroup.

\item[{\rm (c)}] Similarly, if $G = S_n$ or $A_n$ with $n \geqs 5$, then a theorem of Baykalov \cite{Bay} gives $b(G,H) \leqs 5$ for \emph{every} solvable subgroup $H$ of $G$, in which case $d_G(H) \leqs 9$. In this setting, \cite[Conjecture 2]{AB} asserts that $b(G,H) = 2$ for all $n \geqs N$, where $N$ is an unspecified constant. If this is true, then it would imply that every nontrivial solvable subgroup of $S_n$ or $A_n$ has depth $3$ if $n$ is sufficiently large. Moreover, in \cite{AB} the authors speculate that $N=21$ may be the optimal constant, noting that $b(G,H) = 3$ for $G = S_{20}$ and $H = (S_4 \wr S_4) \times S_4$. With the aid of {\sc Magma} \cite{magma}, we can show that $d_G(H) = 5$ in the latter case, which suggests that $n \geqs 21$ may also be optimal for depth-three.
\end{itemize} 
\end{remk}

One of the main motivations for the work in \cite{Bay,B23} is a conjecture of Vdovin  (stated as Problem 17.41(b) in \cite{Kou}). This asserts that if $G$ is a finite group with trivial solvable radical, then $b(G,H) \leqs 5$ for every solvable subgroup $H$ of $G$. There is a reduction of this conjecture to almost simple groups \cite{Vdovin} and the results in \cite{Bay,B23} reduce the problem further to groups of Lie type. For classical groups, there is ongoing work of Baykalov towards a proof of the conjecture, including recent papers \cite{Bay2,Bay3} where he reduces the problem to the orthogonal groups. The conjecture in the special case where $H$ is a solvable maximal subgroup of $G$ is established in \cite{B21}.

This earlier work leads us naturally to the following conjecture on the depth of solvable subgroups of finite groups. Note that in view of Theorem \ref{t:solv} above, the asserted bound on $d_G(H)$ would be best possible. And by the main theorem of \cite{B21}, we know that the conclusion holds whenever $H$ is a solvable \emph{maximal} subgroup of $G$.

\begin{conj}\label{c:vdov}
Let $G$ be a finite group with trivial solvable radical. Then $d_G(H) \leqs 9$ for every solvable subgroup $H$ of $G$.
\end{conj}

\begin{remk}\label{r:q51}
To conclude the introduction, we take the opportunity to resolve a related open problem. Let $G$ be a finite group and let $H$ be a subgroup of $G$ with $d_G(H) = 3$. In this situation, \cite[Question 5.1]{JHH} asks if it is always possible to find elements $x,y \in G$ such that  
\[
{\rm Core}_G(H) = \bigcap_{g \in G} H^g = H \cap H^{x} \cap H^{y},
\] 
which would imply that $b(G,H) \leqs 3$ if $H$ is core-free. With the aid of {\sc Magma} \cite{magma}, it is straightforward to check that this assertion is false. Indeed, there exists a group $G$ of order $2^7$ with a core-free subgroup $H$ of order $8$ such that $d_G(H) = 3$ and $b(G,H) = 4$. Specifically, we can take $G$ to be the group \texttt{SmallGroup(128,560)} in the {\sc Magma} database of small groups and $H$ to be any core-free elementary abelian subgroup of order $8$.
\end{remk}

\noindent \textbf{Acknowledgements.} It is a pleasure to thank Thomas Breuer and Gunter Malle for helpful discussions during the preparation of this paper. In particular, I thank Dr Breuer for his assistance with some of the character table calculations for sporadic groups in Section \ref{s:sporadic}. I also grateful to an anonymous referee for their careful reading of the paper and several helpful comments and suggestions.

\section{Preliminaries}\label{s:prel}

In this section, we present some preliminary results that we will need in the proofs of our main results. In particular, in Section \ref{ss:comp} we briefly discuss some of the main computational methods we use in this paper.

\subsection{Depth}\label{ss:depth}

Let $G$ be a finite group with a subgroup $H$ and consider the inclusion of group algebras $FH \subseteq FG$, where $F$ is a field. For a positive integer $m$, it will be convenient to write $(FG)^m$ for the tensor product
\[
FG \otimes_{FH} \cdots \otimes_{FH} FG \;\; \mbox{(with $m$ factors).}
\]
Following Boltje, Danz and K\"{u}lshammer \cite{BDK}, the inclusion $FH \subseteq FG$ has depth $2n$ for some positive integer $n$ if $(FG)^{n+1}$ is isomorphic as an $(FG,FH)$-bimodule to a direct summand of $\bigoplus_{i=1}^k (FG)^n$ for some $k \geqs 1$. And the inclusion has depth $2n+1$ if the same conclusion holds with respect to $(FH,FH)$-bimodules. Finally, $FH$ has depth $1$ in $FG$ if $FG$ is isomorphic to a direct summand of $\bigoplus_{i=1}^k FH$ as an $(FH,FH)$-bimodule for some $k \geqs 1$. 

\begin{defn}\label{d:depth}
The \emph{depth} of $H$ in $G$, denoted $d_G(H)$, is defined to be the minimal depth of the inclusion of complex group algebras $\mathbb{C}H \subseteq \mathbb{C}G$.
\end{defn}

As we recalled in Section \ref{s:intro}, this notion is well-defined in the sense that every inclusion $FH \subseteq FG$ has depth $n$ for some positive integer $n$ (see \cite[Theorem 4.1]{BDK}). In addition, as noted in \cite[Remark 4.5(a)]{BDK}, the definition is unchanged if we replace $\mathbb{C}$ by any field of characteristic zero.

\begin{prop}\label{p:normal}
We have $d_G(H) \leqs 2$ if and only if $H$ is normal in $G$, with equality if and only if $G \ne HC_G(x)$ for some $x \in H$.
\end{prop}

\begin{proof}
This follows by combining \cite[Corollary 3.2]{KK} with \cite[Theorem 1.7]{BK}.
\end{proof}

In this paper, we are primarily interested in the case where $H$ is a nontrivial core-free subgroup of $G$, in which case $d_G(H) \geqs 3$ by Proposition \ref{p:normal}. In this setting, there are several different ways to compute (or bound) the depth of $H$ and here we briefly recall some of the main methods.

\subsubsection{The inclusion matrix}\label{sss:inc}

The first approach relies on the so-called \emph{inclusion matrix} corresponding to the embedding of $H$ in $G$. In order to give the definition, let 
\[
{\rm Irr}(G) = \{ \chi_1, \ldots, \chi_s\},\;\; {\rm Irr}(H) = \{ \psi_1, \ldots, \psi_r\}
\]
be the irreducible (complex) characters of $G$ and $H$, respectively, and let $M = (m_{ij})$ be the corresponding \emph{inclusion matrix}. This is the $r \times s$ matrix with $m_{ij} = \la \psi_i^G,\chi_j \ra \geqs 0$, where $\psi_i^G$ denotes the character of $G$ induced from $\psi_i$, and $\la \; , \, \ra$ is the usual scalar product on the space of complex class functions on $G$. Set $M^{(1)} = M$ and for each positive integer $\ell$ define 
\begin{equation}\label{e:Mdef}
M^{(2\ell)} = M^{(2\ell-1)}M^{T} \mbox{ and } M^{(2\ell+1)} = M^{(2\ell)}M,
\end{equation}
where $M^T$ is the transpose of $M$. For two integer matrices $A = (a_{ij})$ and $B = (b_{ij})$ of the same size, let us write $A \leqs B$ if $a_{ij} \leqs b_{ij}$ for all $i,j$ (and similarly $A<B$ if $a_{ij} < b_{ij}$ for all $i,j$). In particular, we write $A>0$ if $a_{ij} > 0$ for all $i,j$.

The following result is a key tool for calculating the depth of $H$.

\begin{prop}\label{p:inc}
Let $H$ be a non-normal subgroup of $G$. 
\begin{itemize}\addtolength{\itemsep}{0.2\baselineskip}
\item[{\rm (i)}] The depth $d_G(H)$ is the smallest integer $d$ such that $M^{(d+1)} \leqs eM^{(d-1)}$ for some positive integer $e$.
\item[{\rm (ii)}] If $H$ is core-free, then $d_G(H)$ is the smallest integer $d$ such that $M^{(d-1)} > 0$.
\end{itemize}
\end{prop}

\begin{proof}
Part (i) is due to Burciu, Kadison and K\"{u}lshammer (see \cite[Section 3]{BKK}). Turning to (ii), let us first assume $d$ is the smallest integer such that $M^{(d-1)} > 0$. Clearly, we have $M^{(d+1)} \leqs eM^{(d-1)}$ for some $e \in \mathbb{N}$, and we also note that $M^{(\ell)} >0$ for all $\ell \geqs d$. On the other hand, the minimality of $d$ implies that $M^{(d-2)}$ has at least one zero entry, so there is no positive integer $f$ with $M^{(d)} \leqs fM^{(d-2)}$. So part (i) implies that $d = d_G(H)$ and the result follows.

So to complete the proof, it just remains to show that if $H$ is core-free, then there exists a positive integer $n$ with $M^{(n)}>0$. As explained in Section \ref{sss:kuls} below, we can define an associated graph $\mathcal{K}(G,H)$, whose vertices are labelled by the irreducible characters of $H$. When $H$ is core-free, this graph is connected (see Proposition \ref{p:graph}(i)) and if $\ell$ denotes its diameter, then \cite[Proposition 3.4]{BKK} implies that $(M^{(2\ell)})_{ij}>0$ for all $i \ne j$. Since $M$ is an irredundant matrix (in the sense that every row and column is non-zero) it follows that $(M^{(2)})_{ii} >0$ for all $i$ and we conclude that $M^{(2\ell)}>0$.
\end{proof}

\begin{rem}\label{r:core}
Recall that the \emph{core} of a subgroup $H$ of $G$ is defined by
\[
{\rm Core}_G(H) = \bigcap_{g \in G} H^g
\]
and this is the largest normal subgroup of $G$ contained in $H$. Referring to part (ii) of Proposition \ref{p:inc}, let us observe that if ${\rm Core}_G(H) \ne 1$, then there may not exist a positive integer $n$ with $M^{(n)}>0$. For example, if $G = S_4$ and $H = D_8$, then up to a choice of ordering of the sets ${\rm Irr}(G)$ and ${\rm Irr}(H)$ we have
\[
M = \left(\begin{array}{ccccc}
1 & 0 & 1 & 0& 0 \\
0 & 0 & 0 & 0 & 1 \\
0 & 1 & 1 & 0 & 0 \\
0 & 0 & 0 & 1 & 0 \\
0 & 0 & 0 & 1 & 1 
\end{array}\right)
\]
and one can check that the $(2,1)$-entries of $M^{(2\ell)}$ and $M^{(2\ell+1)}$ are zero for all $\ell \geqs 1$. But we see that $M^{(5)} \leqs 4M^{(3)}$ and so part (i) of Proposition \ref{p:inc} implies that $d_G(H) = 4$. 
\end{rem}

\subsubsection{The K\"{u}lshammer graph}\label{sss:kuls}

Another way to study $d_G(H)$ is in terms of a simple undirected graph, which we refer to here as the \emph{K\"{u}lshammer graph}, denoted by $\mathcal{K}(G,H)$. The vertices of this graph are labelled by the irreducible characters of $H$, with an edge between distinct characters $\a,\b$ (denoted $\a \sim \b$) if $\la \a^G,\b^G \ra \ne 0$. In other words, $\a \sim \b$ if and only if the induced characters $\a^G$ and $\b^G$ have a common irreducible constituent. We then define $d(\a,\b)$ to be the minimal positive integer $m$ such that there exists a path 
\[
\a = \a_0 \sim \a_1 \sim \cdots \sim \a_m = \b
\]
in $\mathcal{K}(G,H)$. If no such path exists, then we set $d(\a,\b) = -\infty$. Note that $d(\a,\a) = 0$ for all $\a \in {\rm Irr}(H)$. 

In addition, if $\chi \in {\rm Irr}(G)$ and $\Lambda \subseteq {\rm Irr}(H)$ is the set of  irreducible constituents of the restricted character $\chi_{H}$, then we set
\[
\mu_H(\chi) = \max_{\a \in {\rm Irr}(H)} \left(\min_{\b \in \Lambda} d(\a,\b) \right),
\]
which coincides with the maximal distance in $\mathcal{K}(G,H)$ between $\Lambda$ and an irreducible character of $H$.

\begin{prop}\label{p:graph}
Let $H$ be a nontrivial core-free subgroup of $G$. Then the following hold:
\begin{itemize}\addtolength{\itemsep}{0.2\baselineskip}
\item[{\rm (i)}] The graph $\mathcal{K}(G,H)$ is connected. 

\item[{\rm (ii)}] If $m$ is the diameter of $\mathcal{K}(G,H)$, then $d_G(H) = 2m$ or $2m+1$.

\item[{\rm (iii)}] If $\max\{ \mu_H(\chi) \,:\, \chi \in {\rm Irr}(G)\} = \ell$, then $d_G(H) = 2\ell+1$ or $2\ell+2$.

\item[{\rm (iv)}] We have $d_G(H) =3$ if and only if $\la \a^G, \b^G \ra \ne 0$ for all $\a, \b \in {\rm Irr}(H)$. 

\item[{\rm (v)}] We have $d_G(H) \geqs 5$ if $\la \a^G,(\chi_H)^G \ra = 0$ for some $\a \in {\rm Irr}(H)$ and $\chi \in {\rm Irr}(G)$ with $\chi_H$ irreducible.

\item[{\rm (vi)}] If $K \leqs H$, then $d_G(K) \leqs d_G(H)$.
\end{itemize}
\end{prop}

\begin{proof}
Part (i) is \cite[Proposition 2.1]{BHHK}, which implies that if $H$ is a subgroup of $G$, then $\mathcal{K}(G,H)$ is connected if and only if $H$ is core-free. Parts (ii) and (iii) follow from Theorems 3.6 and 3.10 in \cite{BKK}, respectively.

Next consider (iv) and recall that $d_G(H) \geqs 3$ by Proposition \ref{p:normal}. If $\la \a^G, \b^G \ra \ne 0$ for all $\a, \b \in {\rm Irr}(H)$ then $\mathcal{K}(G,H)$ is complete, so this graph has diameter $1$ and part (ii) implies that $d_G(H) = 3$. And conversely, if $d_G(H) = 3$ then part (ii) implies that $\mathcal{K}(G,H)$ is complete and the result follows.

Turning to (v), and referring to the above definition of $\mu_H(\chi)$, we note that $\Lambda = \{ \chi_H\}$ and
\[
\mu_H(\chi) = \max_{\a \in {\rm Irr}(H)} d(\a,\chi_H).
\]
Since $\la \a^G,(\chi_H)^G \ra = 0$ for some $\a \in {\rm Irr}(H)$, it follows that $d(\a,\chi_H) \geqs 2$ and thus $\mu_H(\chi) \geqs 2$. Therefore, $d_G(H) \geqs 5$ by part (iii).

Finally, suppose $K$ is a subgroup of $H$, so both $K$ and $H$ are core-free. If $K=1$, then $d_G(K) = 1 < d_G(H)$, so we are free to assume that $K$ is nontrivial.

First assume $d_G(H) = 2m+1$ is odd, so part (ii) implies that $\mathcal{K}(G,H)$ has diameter $m$ and it suffices to show that $\mathcal{K}(G,K)$ has diameter at most $m$. So let $\gamma,\delta \in {\rm Irr}(K)$ be distinct irreducible characters of $K$ and fix $\a,\b \in {\rm Irr}(H)$ such that $\gamma$ and $\delta$ are irreducible constituents of the restricted characters $\a_K$ and $\b_K$. Set $d(\a,\b) = k$ with respect to $\mathcal{K}(G,H)$ and  note that $k \leqs m$, say 
\[
\a = \a_0 \sim \a_1 \sim \cdots \sim \a_k = \b
\]
is a path in $\mathcal{K}(G,H)$, in which case $\la \a_i^G, \a_{i+1}^G \ra \ne 0$ for all $i$. Now let $\gamma_i$ be an irreducible constituent of the restriction of $\a_i$ to $K$,  setting $\gamma_0 = \gamma$ and $\gamma_k = \delta$. Then $\la \gamma_i^G, \gamma_{i+1}^G \ra \ne 0$ for all $i$ and thus $d(\gamma,\delta) \leqs k$ in $\mathcal{K}(G,K)$. This shows that the diameter of $\mathcal{K}(G,K)$ is at most $m$ and thus part (ii) gives $d_G(K) \leqs 2m+1 = d_G(H)$. 

A very similar argument, working with part (iii), gives the same conclusion when $d_G(H) = 2m$ is even (specifically, $\mu_K(\chi) \leqs \mu_H(\chi)$ for all $\chi \in {\rm Irr}(G)$).
We omit the details.
\end{proof}

\subsubsection{Bases}\label{sss:bases}

Recall that if $G \leqs {\rm Sym}(\O)$ is a transitive permutation group on a finite set $\O$ with point stabilizer $H$, then a \emph{base} for $G$ is a subset $B$ of $\O$ with trivial pointwise stabilizer (that is to say, the identity element is the only element in $G$ fixing every point in $B$). We then define the \emph{base size} of $G$, denoted $b(G,H)$, to be the minimal size of a base. This is a classical and intensively studied invariant in permutation group theory and there have been many remarkable advances in this area in recent years, especially in the setting where $G$ is an almost simple primitive group (we refer the reader to Section \ref{s:intro} for more details). 

Notice that the point stabilizer $H$ is a core-free subgroup of $G$ and we can interpret $b(G,H)$ as the smallest positive integer $b$ such that there exist elements $x_1, \ldots, x_b$ in $G$ with
\[
\bigcap_{i=1}^b H^{x_i} = 1.
\]
Given this observation, the following result is an immediate consequence of \cite[Theorem 6.9]{BKK}.

\begin{prop}\label{p:bases}
Let $G$ be a finite group, let $H$ be a core-free subgroup of $G$ and let $b(G,H)$ be the base size of $G$ with respect to the action on $G/H$. 
\begin{itemize}\addtolength{\itemsep}{0.2\baselineskip}
\item[{\rm (i)}] We have $d_G(H) \leqs 2b(G,H)-1$.
\item[{\rm (ii)}] In particular, if $b(G,H)=2$ then $d_G(H) = 3$.
\end{itemize}
\end{prop}

The connection between the depth $d_G(H)$ of a subgroup and the base size $b(G,H)$ encapsulated in part (i) of Proposition \ref{p:bases} has been used effectively by several authors to study the depth of subgroups in some specific families of simple groups. Here the rank one groups of Lie type have been the main focus and we refer the reader to the sequence of papers \cite{F1,F2,HHP1,HHP2,JHH}. But the literature on bases for almost simple groups is far more extensive and one of our main aims in this paper is to exploit  some of the main advances, which will allow us to substantially extend our understanding of the depth of subgroups of almost simple groups in far greater generality than has been achieved up to date.

A key tool for studying bases is a powerful probabilistic method originally introduced by Liebeck and Shalev \cite{LSh99} in their proof of a conjecture of Cameron and Kantor on base sizes for a certain family of almost simple primitive permutation groups (the so-called \emph{non-standard} groups, which we briefly highlighted in Section \ref{s:intro} in relation to Theorem \ref{t:ns}). In view of the bound in Proposition \ref{p:bases}(i), this approach can also be used to obtain effective upper bounds on the depth of a subgroup. 

Let $G$ be a finite group, let $H$ be a core-free subgroup and consider the natural action of $G$ on the set $\O = G/H$ of cosets of $H$ in $G$. Given a positive integer $t$, let 
\[
P(G,t) = \frac{|\{ (\a_1, \ldots, \a_t) \in \O^t \,:\, \bigcap_i G_{\a_i} \ne 1\}|}{|\O^t|}
\]
be the probability that a uniformly random $t$-tuple of points in $\O$ does not form a base for $G$. Let $\chi$ be the corresponding permutation character of $G$ and define
\begin{equation}\label{e:QGt}
Q(G,t) =  \sum_{i=1}^k |x_i^G| \cdot \left(\frac{\chi(x_i)}{n}\right)^t,
\end{equation}
where $n = |\O|= |G:H| = \chi(1)$ and $\{x_1, \ldots, x_k\}$ is a complete set of representatives of the conjugacy classes in $G$ of elements of prime order. 

\begin{prop}\label{p:prob}
Let $G$ be a finite group and let $H$ be a nontrivial core-free subgroup.
\begin{itemize}\addtolength{\itemsep}{0.2\baselineskip}
\item[{\rm (i)}] If $Q(G,t)<1$ then $d_G(H) \leqs 2t-1$. 
\item[{\rm (ii)}] In particular, if $Q(G,2)<1$ then $d_G(H) = 3$.
\end{itemize}
\end{prop}

\begin{proof}
First observe that a $t$-tuple $(\a_1, \ldots, \a_t) \in \O^t$ is not a base for $G$ if and only if there exists an element $x \in G$ of prime order fixing each $\a_i$. Now, the probability that a given element $x \in G$ fixes a uniformly random point in $\O$ is given by $\chi(x)/n$, so the probability it fixes a random $t$-tuple of points is equal to $(\chi(x)/n)^t$. Therefore,
\[
P(G,t) \leqs \sum_{x \in \mathcal{P}} \left(\frac{\chi(x)}{n}\right)^t = Q(G,t),
\]
where $\mathcal{P} = \bigcup_i x_i^G$ is the set of prime order elements in $G$. The result now follows from Proposition \ref{p:bases} since $b(G,H) \leqs t$ if and only if $P(G,t)<1$.  
\end{proof}

\subsection{Almost simple groups}\label{ss:as}

Recall that a finite group $G$ is \emph{almost simple} if the socle of $G$, denoted $G_0$, is a nonabelian simple group. In particular, by identifying $G_0$ with its group of inner automorphisms, we have $G_0 \normeq G \leqs {\rm Aut}(G_0)$.

In this paper, we are primarily interested in studying the depth of core-free subgroups of almost simple groups, which allows us to apply many of the results and techniques presented in Section \ref{ss:depth}. But let us begin by considering the depth of the socle of an almost simple group.

\begin{prop}\label{p:socle}
Let $G$ be an almost simple group with socle $G_0$. Then $d_G(G_0) = 1+\e$, where $\e = 0$ if $G = G_0$, otherwise $\e=1$.
\end{prop}

\begin{proof}
In view of Proposition \ref{p:normal}, we may assume $G \ne G_0$ and it suffices to show that there exists an element $x \in G_0$ such that $G \ne G_0C_G(x)$. 

Suppose $G = G_0C_G(x)$ for all $x \in G_0$. Then $x^G = x^{G_0}$ and thus every conjugacy class in $G_0$ is $G$-invariant. But this contradicts a theorem of Feit and Seitz \cite[Theorem C]{FS} and we conclude that $d_G(G_0) = 2$.
\end{proof}

As mentioned in Section \ref{s:intro}, there are several results in the literature concerning the depth of subgroups of certain almost simple groups (namely, symmetric and alternating groups, together with the rank one simple groups of Lie type). A summary of these results is presented in Table \ref{tab:as}.

\begin{rem}\label{r:tabas}
Let us record some comments on the pairs $(G,H)$ appearing in Table \ref{tab:as}.
\begin{itemize}\addtolength{\itemsep}{0.2\baselineskip}
\item[{\rm (a)}] In each case, $H$ is a proper nontrivial subgroup of $G$. For the groups of Lie type, we write $q = p^f$ with $p$ a prime, and we write ${\rm Syl}_p(G)$ for the set of Sylow $p$-subgroups of $G$ (so in each case appearing in the table, $N_G(S)$ is a Borel subgroup of $G$ for each $S \in {\rm Syl}_p(G)$).

\item[{\rm (b)}] Suppose $G = S_n$ and $H = S_k \times S_{n-k}$, with $2k \leqs n$. As indicated in the table, if $n \geqs (k^2+2)/2$ then the exact depth $d_G(H)$ is known. In addition, $d_G(H)$ has been computed for all $(n,k)$ with $k \leqs 10$ and $2k \leqs n \leqs 39$ (see \cite[Theorem 3.10]{FKR}). And in the remaining cases, the authors of \cite{FKR} present upper and lower bounds on $d_G(H)$ (for example, see \cite[Corollary 3.2]{FKR} for a general lower bound). However, it remains an open problem to determine $d_G(H)$ precisely in all cases.

\item[{\rm (c)}] A complete analysis of the depth of subgroups of $G = {\rm L}_2(q)$ is given by Fritzsche in \cite{F1}. It turns out that there are two natural families of subgroups with depth $5$, while all other proper nontrivial subgroups have depth $3$ if $q \geqs 13$. There are a handful of additional subgroups $H$ with $d_G(H) > 3$ when $q \leqs 11$, namely
\[
\mbox{ } \hspace{10mm} 
d_G(H) = \left\{ \begin{array}{ll}
5 & \mbox{if $H = S_4$ and $q \in \{4,5,7\}$, or if $H = A_5$ and $q \in \{9,11\}$} \\
4 & \mbox{if $H = S_4$ and $q=9$.} 
\end{array}\right.
\]

\item[{\rm (d)}] The depth of every maximal subgroup of $G = {}^2G_2(q)$ with $q \geqs 27$ is computed in \cite{HHP2}. As far as the author is aware, this result has not been extended to all proper subgroups of $G$. In addition, note that ${}^2G_2(3)' \cong {\rm L}_2(8)$, so the depth of every subgroup of ${}^2G_2(3)'$ can be read off from \cite{F1}.
\end{itemize}
\end{rem}
 
{\scriptsize
\begin{table}
\[
\begin{array}{lllll} \hline
G & H  & \mbox{Conditions} & d_G(H) & \mbox{References} \\ \hline
S_n & S_{n-k} & k = 1 & 2n-3 & \mbox{\cite[6.19, A.2]{BKK}} \\
& & 1 \leqs k < n & 2\lfloor (n-2)/k\rfloor +1 & \mbox{\cite[2.1]{FKR}} \\
A_n & A_{n-k} & k=1 & 2n-2\lceil \sqrt{n-1} \rceil-1 & \mbox{\cite[A.5]{BKK}} \\
S_n & S_k \times S_{n-k} & k \geqs 2, \; n \geqs (k^2+2)/2 & 2\lceil 2(n-1)/(k+1)\rceil-1 & \mbox{\cite[3.7]{FKR}} \\
& & n = 2k & 2\lceil \log_2k\rceil+1 &  \\ 
& & & & \\
{\rm L}_2(q) & D_{2(q+1)} & \mbox{$q$ even} & 5 & \mbox{\cite[2.22]{F1}} \\
& S{:}C_r \leqs N_G(S) & S \in {\rm Syl}_p(G),\, r > 1 & 5 &  \\
& \mbox{other} & q \geqs 13 & 3 &  \\
& & & & \\
{}^2B_2(q) & S{:}C_r \leqs N_G(S) & S \in {\rm Syl}_2(G),\, r > 1 & 5 & \mbox{\cite[1.3]{JHH}, \cite[5.1]{HHP1}} \\
& \mbox{other} & & 3 &  \\
& & & & \\
{}^2G_2(q) & N_G(S) & q \geqs 27,\, S \in {\rm Syl}_3(G) & 5 & \mbox{\cite[1.4]{HHP2}} \\
& \mbox{other maximal} & & 3 &  \\
\hline
\end{array}
\]
\caption{Earlier depth results for almost simple groups}
\label{tab:as}
\end{table}
}

\subsection{Computational methods}\label{ss:comp}

Computational methods play a key role in the proofs of several of our main results and we use both {\sc Magma} \cite{magma} (version V2.28-19) and \textsf{GAP} \cite{GAP} (version 4.13.1). We also make use of the extensive information on character tables and class fusions, which is available in the \textsf{GAP} Character Table Library \cite{GAPCTL}. In this section, we provide a brief summary of our main techniques. 

Throughout this section, $G$ will denote an almost simple group with socle $G_0$ and $H$ will be a nontrivial core-free subgroup of $G$.

\subsubsection{{\sc Magma} computations}\label{sss:mag}

Let us explain how we can compute $d_G(H)$ with the aid of {\sc Magma} \cite{magma} in a typical situation that arises in this paper. 

First we use the function \texttt{AutomorphismGroupSimpleGroup} in order to work with a suitable permutation representation of $G$ and we embed $H$ in a maximal subgroup $H_1$ of $G$, which we can construct via the function \texttt{MaximalSubgroups}. If $|H_1|^2 < |G|$, then we can use random search to seek an element $x \in G$ with $H_1 \cap H_1^x = 1$; if such an element exists, then $b(G,H_1) = 2$ and thus $d_{G}(H_1) = 3$ by Proposition \ref{p:bases}(ii), which in turn implies that $d_G(H) = 3$ by Proposition \ref{p:graph}(vi). So let us assume that this random search is inconclusive (after a fixed number of attempts), or that $|H_1|$ satisfies the inequality $|H_1|^2 \geqs |G|$. 

To proceed in this situation, we use the \texttt{CharacterTable} function, which is an implementation of an algorithm of Unger \cite{Unger}, to compute the character tables of both $G$ and $H_1$. Then by inducing the irreducible characters of $H_1$, we can calculate the inclusion matrix $M$ for the embedding of $H_1$ in $G$ and then it is a routine computation to determine the minimal $d_1$ such that $M^{(d_1-1)} > 0$ in the notation of Proposition \ref{p:inc}. Then by part (ii) of this proposition, we deduce that $d_G(H_1) = d_1$ and thus $d_G(H) \leqs d_1$ by Proposition \ref{p:graph}(vi). And as above, we conclude that $d_G(H) = 3$ if $d_1 = 3$.

Finally, let us assume $H \ne H_1$ and $d_1 \geqs 4$. In this situation, we embed $H$ in a maximal subgroup $H_2$ of $H_1$ and we proceed as above in order to compute $d_2 = d_G(H_2)$. If $H = H_2$, or if $d_2 = 3$, then we conclude that $d_G(H) = d_2$, otherwise we embed $H$ in a maximal subgroup of $H_2$ and we repeat again.

\begin{rem}\label{r:mag}
In practice, for the groups $G$ we will be working with computationally, we will usually find that every $3$-maximal subgroup of $G$ has depth $3$ and so the iterative process described above terminates after at most two or three iterations. And we can use essentially the same approach to determine all the proper subgroups $H$ of $G$ (up to conjugacy in $G$) with $d_G(H) \geqs 4$, repeatedly taking maximal subgroups to descend deeper in to the subgroup lattice of $G$ until we reach a layer where every subgroup has depth at most $3$.
\end{rem}
 
\begin{rem}\label{r:web}
There are a small number of cases we need to consider where the {\sc Magma} function \texttt{MaximalSubgroups} highlighted above is ineffective. For example, this is the case when $G$ is Fischer's sporadic group ${\rm Fi}_{24}'$. To handle this situation, we can work with explicit generators provided in the Web Atlas \cite{WebAt}, which allows us to construct a representative of the conjugacy classes of maximal subgroups $H$ we need to consider. In each case, the generators are presented as words in a pair of standard generators for $G$. Here it is perhaps worth noting that for $G = {\rm Fi}_{24}'$ we first construct $J = N_L(H) = H.2$ in $L = {\rm Fi}_{24} = G.2$, using the generators in \cite{WebAt}, and then we obtain $H$ as $J \cap G$, which allows us to construct all the maximal subgroups of $G$ we need to consider (see Case 1 in the proof of Proposition \ref{p:3}).
\end{rem}

\subsubsection{\textsf{GAP} computations}\label{sss:gap}

We will also work with \textsf{GAP} \cite{GAP} and the Character Table Library \cite{GAPCTL}. Let us suppose that the character tables of $G$ and $H$ are available in \cite{GAPCTL} (for example, we may be able to access the latter via the functions \texttt{Maxes} or \texttt{NamesOfFusionSources}) together with the fusion map from $H$-classes to $G$-classes. In this situation, we can induce characters of $H$ to $G$, which allows us to calculate the corresponding inclusion matrix $M$. We can then determine the minimal $d$ such that $M^{(d-1)} >0$, which coincides with $d_G(H)$ by Proposition \ref{p:inc}(ii).

\begin{rem}\label{r:gap}
There are several cases where we use \textsf{GAP} to compute $d_G(H)$, for which it would be impossible to work directly with {\sc Magma}. For example, using \textsf{GAP} as above we can show that $d_G(H) = 5$ when $G = \mathbb{M}$ and $H = 2.\mathbb{B}$, which is not a case that we could handle using {\sc Magma}.
\end{rem}

The above approach can also be effective when the fusion map from $H$-classes to $G$-classes is not available in \cite{GAPCTL}. In this situation, we can use the function \texttt{PossibleClassFusions} to obtain a collection of candidate fusion maps $f_1, \ldots, f_s$ (one of which is the correct map). Then for each candidate map $f_i$ and each irreducible character $\a$ of $H$, we can use \texttt{InducedClassFunctionsByFusionMap} to calculate a candidate for the corresponding induced character $\a_i = \a^G$, which in turn allows us to compute a candidate inclusion matrix $M_i$ (one for each possible fusion map). If $d_i$ is minimal such that $M_i^{(d_i-1)} >0$, then 
\[
\min_i d_i \leqs d_G(H) \leqs \max_i d_i.
\]
In particular, if all the $d_i$ are equal, then $d_G(H) = d_1$.

In some cases, we will also use \textsf{GAP} in order to implement the probabilistic approach for bounding $d_G(H)$, which is featured in Proposition \ref{p:prob}. More precisely, suppose we have access to the character table of $G$ and we are able to compute the permutation character $\chi$ corresponding to the action of $G$ on $G/H$. Then for each integer $t \geqs 2$, we can  evaluate the expression $Q(G,t)$ in \eqref{e:QGt}, recalling that $d_G(H) \leqs 2t-1$ if $Q(G,t) < 1$. Moreover, if $Q(G,2) < 1$ then $d_G(H) = 3$. And as above, we may be in a situation where we can only calculate a collection of candidate permutation characters $\a_1, \ldots, \a_s$ for the action of $G$ on $G/H$. Then for each $\a_j$ we can compute 
\[
Q_j(G,t) = \sum_{i=1}^k |x_i^G| \cdot \left(\frac{\a_j(x_i)}{n}\right)^t,
\]
where $n = |G:H|$ and $x_1, \ldots, x_k$ is a set of representatives of the conjugacy classes in $G$ of elements of prime order. And of course, if $Q_j(G,t) < 1$ for all $j$, then $d_G(H) \leqs 2t-1$.

\begin{rem}\label{r:com}
Some cases will require a combination of methods. For example, suppose the character table of $G$ is available in \textsf{GAP}, but it is not feasible to work directly with $G$ in {\sc Magma} (for example, $G$ could be the Baby Monster or the Monster sporadic group). And let us assume the character table of $H$ is not available in \cite{GAP}, but we have a permutation or matrix representation of $H$, which allows us to calculate the character table of $H$ in {\sc Magma}. In this situation, we first use {\sc Magma} to construct the character table of $H$ and we then use the \textsf{GAP} function \texttt{GAPTableOfMagmaFile} to convert the {\sc Magma} output into \textsf{GAP}-readable form. This allows us to access the character table of $H$ in \textsf{GAP} and we can then use \texttt{PossibleClassFusions} to determine a set of candidate inclusion matrices as above (or a set of candidate permutation characters), which allows us to proceed as before. For example, see the proof of Case 2 in the proof of Proposition \ref{p:3}, where we use this approach to study the depth of certain subgroups of a maximal subgroup $2^{11}{:}{\rm M}_{24}$ of ${\rm J}_4$.
\end{rem}

\section{Proof of Theorems \ref{t:ns}-\ref{t:class}}\label{s:proofs}

The purpose of this section is to prove Theorems \ref{t:ns}-\ref{t:class}, which rely heavily on the extensive literature on bases for almost simple primitive permutation groups.

\subsection{Proof of Theorem \ref{t:ns}}

Define $G$, $G_0$ and $H$ as in the statement of the theorem. We appeal initially to the proof of Cameron's conjecture, which is presented in the series of papers \cite{B07,BGS,BLS,BOW}. In particular, we have $b(G,H) \leqs 7$, with equality if and only if $G = {\rm M}_{24}$ and $H = {\rm M}_{23}$, so the bound in part (iv) follows immediately from Proposition \ref{p:bases}(i). Similarly, part (i) follows by combining \cite[Corollary 1.2]{BGS} with the bound on $d_G(H)$ in Proposition \ref{p:bases}(i), noting that the special case $(G,H) = (S_6,S_5)$ with $H$ acting primitively on $\{1,\ldots, 6\}$ is recorded as $(S_6,{\rm PGL}_2(5))$ in \cite[Table 1]{BGS}.

Next assume $G_0$ is a sporadic group. Since $d_G(H) \leqs 9$ if $b(G,H) \leqs 5$, we only need to consider the groups with $b(G,H) = 6$ or $7$, which we can read off from \cite{BOW}:
\begin{equation}\label{e:list1}
({\rm M}_{23},{\rm M}_{22}),\; ({\rm M}_{24},{\rm M}_{23}),\; ({\rm Co}_3, {\rm McL}.2),\; ({\rm Co}_2, {\rm U}_6(2).2),\; ({\rm Fi}_{22}.2, 2.{\rm U}_6(2).2).
\end{equation}
In each of these cases, we can access the character tables of $G$ and $H$ in the \textsf{GAP} Character Table Library \cite{GAPCTL} and this allows us to compute the corresponding inclusion matrix in each case (see Sections \ref{sss:inc} and \ref{sss:gap}). From here, it is a straightforward exercise to determine the minimal integer $d$ such that $M^{(d-1)}>0$ and by appealing to Proposition \ref{p:inc}(ii) we deduce that $d_G(H) = d$. In this way, one can check that $d_G(H) \leqs 11$ for each pair in \eqref{e:list1}, with equality if and only if $(G,H) = ({\rm M}_{24},{\rm M}_{23})$. This establishes part (ii) of Theorem \ref{t:ns}.

Part (iii) is also very similar. Here $G_0$ is a classical group and the main theorem of \cite{B07} gives $b(G,H) \leqs 5$, with equality if and only if $(G,H) = ({\rm U}_6(2).2, {\rm U}_4(3).2^2)$. So it just remains to show that $d_G(H) = 7$ in the latter case. We can do this using {\sc Magma}, following the approach described in Section \ref{sss:mag}.
More specifically, we first use the function \texttt{AutomorphismGroupSimpleGroup} to construct ${\rm Aut}(G_0) = G_0.S_3$ as a permutation group of degree $693$. We then identify $G = G_0.2$ as the unique index-three subgroup of ${\rm Aut}(G_0)$, up to conjugacy, and we use the function \texttt{MaximalSubgroups} to construct a conjugate of $H$ as a subgroup of $G$ in this representation. We can construct the character tables of $G$ and $H$, which as before allows us to determine the associated inclusion matrix $M$. It is then routine to check that $n = 6$ is the smallest integer with $M^{(n)}>0$ and we conclude that $d_G(H) = 7$ by applying Proposition \ref{p:inc}(ii).

\vs

This completes the proof of Theorem \ref{t:ns}.

\subsection{Proof of Theorem \ref{t:sn}}

Let $G$ be an almost simple group with socle $G_0 = A_n$ and let $H$ be a core-free subgroup of $G$ such that $H \cap G_0$ acts primitively on $\{1, \ldots, n\}$. Assume for now that $n \ne 6$, so $G = S_n$ or $A_n$, and let $K$ be maximal among core-free subgroups of $G$ containing $H$. Note that if $G = G_0$, then $K$ is a maximal subgroup of $G$, whereas $K$ is maximal in one of $G_0$ or $G$ when $G = S_n$ (for example, if $G = S_8$ and $H = {\rm AGL}_3(2)$, then $K = H$ is maximal in $G_0$, but not $G$). If we set $L = G_0$ or $G$ accordingly, then 
\begin{equation}\label{b:bds}
b(G,H) \leqs b(G,K) \leqs b(L,K)
\end{equation}
and we deduce that $d_G(H) = 3$ if $b(L,K) = 2$. Therefore, in view of \cite[Theorem 1.1]{BGS},  it follows that $d_G(H) = 3$ if $n \geqs 13$ and so the proof of Theorem \ref{t:sn} is reduced to the low degree groups with $n \leqs 12$.

The proof can now be completed with the aid of {\sc Magma}, working systematically through the short list of groups that require further attention. For example, suppose $G = S_n$ with $5 \leqs n \leqs 12$, and let $\mathcal{M}$ be a complete set of representatives of the $G$-classes of core-free maximal subgroups $K$ in $G$ and $G_0$, with the additional property that $K \cap G_0$ acts primitively on $\{1, \ldots, n\}$. Working with the natural permutation representation of $G$ in {\sc Magma}, we construct the character tables of $G$ and $K$, which we then use to determine $d_G(K)$ from the corresponding inclusion matrix. If $d = d_G(K) \geqs 4$, then we record the triple $(G,K,d)$ and we repeat the process, taking a set of representatives of the maximal subgroups of $K$ in order to descend deeper in to the subgroup lattice of $G$. We continue until we reach a subgroup with the property that all maximal subgroups have depth $3$ (at this point, in view of Proposition \ref{p:graph}(vi), there is no need to descend any further).

This iterative process is straightforward to implement in {\sc Magma} and it runs very efficiently for the groups we need to consider with $n \leqs 12$. An entirely similar analysis can be carried out if $G = A_n$, and similarly if $n = 6$ and $G = A_6.2 \ne S_6$ or $A_6.2^2$.

\vs

This completes the proof of Theorem \ref{t:sn}.

\subsection{Proof of Theorem \ref{t:sn2}}

First assume $G = S_n$ with $n \geqs 3$ and let $H$ be a subgroup of $G$. The cases $n \in \{3,4,5\}$ can be handled directly, so we may assume $n \geqs 6$, which means that $G$ is almost simple. For example, if $G = S_4$ and $H = D_8$, then $d_G(H) = 4$, as noted in Remark \ref{r:core}. In addition, recall that $d_G(H) = 1$ if $H = 1$ or $G$, while Proposition \ref{p:socle} gives $d_G(H) = 2$ if $H = A_n$. So we may assume $H$ is core-free. By \cite[Theorem 1(i)]{AB}, we have $b(G,H) \leqs n-1$ and thus $d_G(H) \leqs 2n-3$ by Proposition \ref{p:bases}(i), with equality if $H = S_{n-1}$ by \cite[Theorem 6.19]{BKK}.

It remains to show that $d_G(H) < 2n-3$ when $H \ne S_{n-1}$. 
As in \eqref{b:bds} we have $b(G,H) \leqs b(L,K)$, where $H \leqs K < L$ and $K$ is a core-free maximal subgroup of $L = S_n$ or $A_n$. In view of Proposition \ref{p:bases}(i) it suffices to show that $b(G,H) \leqs n-2$.

If $L = A_n$ then \cite[Theorem 1(i)]{AB} gives $b(L,K) \leqs n-2$, so we may assume $L = S_n$. If $K \ne S_{n-1}$, then  by inspecting the relevant literature on base sizes for primitive actions of $S_n$ (see Remark \ref{r:sn}, for example) we deduce that $b(L,K) \leqs n-2$. Now suppose $K = S_{n-1}$. If $H = A_{n-1}$, then it is easy to see that $b(G,H) = n-2$ and the result follows. Otherwise, $H$ is a core-free subgroup of $K$ and $b(K,H) \leqs n-2$ by \cite[Theorem 1(i)]{AB}, which in turn implies that $b(G,H) \leqs n-2$ as before.

For the remainder, let $G = A_n$. The desired result for $n \leqs 7$ is easy to check, so we may assume $n \geqs 8$. Set 
\[
f(n) = 2n-2\lceil \sqrt{n-1}\rceil -1.
\]
To begin with, let us assume $H$ is a maximal subgroup of $G$. We claim that $d_G(H) \leqs f(n)$, with equality if and only if $H = A_{n-1}$.

If $H$ acts primitively on $\{1, \ldots, n\}$, then by appealing to Theorem \ref{t:sn} it is straightforward to check that $d_G(H) < f(n)$ for all $n \geqs 8$. And if $H$ acts transitively and imprimitively, then the main theorem of \cite{MS} yields $b(G,H) \leqs \lfloor 2\log_2n +2\rfloor$ and thus $d_G(H) \leqs 2\lfloor 2\log_2n +2\rfloor-1$. This upper bound gives $d_G(H) < f(n)$ for all $n \geqs 14$. And for $8 \leqs n \leqs 13$ it is straightforward to check that $d_G(H) < f(n)$ in every case.

Next suppose $H = (S_k \times S_{n-k}) \cap G$ is an intransitive maximal subgroup of $G$, where $1 \leqs k < n/2$. 
If $k=1$, then $H = A_{n-1}$ and \cite[Proposition A.5]{BKK} gives $d_G(H) = f(n)$. Now assume $k \geqs 2$. We claim that $d_G(H) < f(n)$. 

The groups with $n \leqs 12$ can be checked directly, so let us assume $n \geqs 13$. If $k \leqs \sqrt{n}$, then \cite[Theorem 3.2]{Halasi} states that 
\[
b(G,H) \leqs \left\lceil \frac{2n-2}{3} \right\rceil
\]
and by applying the upper bound on $d_G(H)$ in Proposition \ref{p:bases}(i) we get $d_G(H) < f(n)$ for all $n \geqs 13$. Now assume $k>\sqrt{n}$. Here \cite[Corollary 4.3]{Halasi} provides the upper bound
\begin{equation}\label{e:b1}
b(G,H) \leqs \lceil \log_tn\rceil (t-1),
\end{equation}
where $t = \lceil n/k \rceil$. As a consequence, it follows that
\[
d_G(H) \leqs 2\lceil \log_2n \rceil \cdot (\left\lceil \sqrt{n}\right\rceil-1)-1
\]
and one can check that this upper bound is less than $f(n)$ for all $n \geqs 44$. To handle the remaining groups we work directly with the upper bound in \eqref{e:b1}, which gives 
\[
d_G(H) \leqs 2\lceil \log_tn\rceil (t-1) - 1 = g(n,k),
\]
and once again it is easy to check that $g(n,k) < f(n)$ for all $8 \leqs n \leqs 43$ and all $\sqrt{n}<k<n/2$.

In view of Proposition \ref{p:graph}(vi), to complete the proof for $G = A_n$ it suffices to show that $d_G(H) < f(n)$ when $H<K<G$ and $H$ is a maximal subgroup of $K = A_{n-1}$. If $H = A_{n-2}$ then $b(G,H) = \lfloor (n-1)/2 \rfloor$ and thus $d_G(H) \leqs n-2 < f(n)$. In each of the remaining cases, $H$ either acts primitively on $\{1, \ldots, n-1\}$, or it is transitive and imprimitive, or it is intransitive of the form $(S_k \times S_{n-1-k}) \cap K$ for some $2 \leqs k <(n-1)/2$. As above, we can now work with a suitable upper bound on $b(K,H)$, which immediately yields an upper bound on $b(G,H)$ and hence an upper bound on $d_G(H)$ via Proposition \ref{p:bases}(i). In this way, it is straightforward to verify the bound $d_G(H)<f(n)$ and the result follows.

\vs

This completes the proof of Theorem \ref{t:sn2}.

\begin{rem}
Note that Theorem \ref{t:sn2} is stated in terms of $S_n$ and $A_n$. For completeness, let us record that if $G$ is any almost simple group with socle $A_6$, then
\[
\max\{d_G(H) \,:\, H \leqs G\} = \left\{\begin{array}{ll}
9 & \mbox{if $G = S_6$} \\
7 & \mbox{if $G = A_6.2^2$} \\
5 & \mbox{if $G = A_6$, ${\rm PGL}_2(9)$ or ${\rm M}_{10}$.}
\end{array}\right.
\]
\end{rem}

\subsection{Proof of Theorem \ref{t:class}}

As in the theorem, let $G$ be an almost simple classical group with socle $G_0$ and let $r$ be the (untwisted) Lie rank of $G_0$ (so $r$ is the rank of the ambient simple algebraic group). Let $H$ be a core-free subgroup of $G$ and let $K$ be a maximal core-free subgroup of $G$ containing $H$. Note that $K$ is a maximal subgroup of some almost simple group $L \leqs G$ with socle $G_0$. Then \eqref{b:bds} holds and it suffices to show that $b(L,K) \leqs 2r+15$.

If $K$ is a \emph{non-subspace} subgroup of $L$ (see \cite[Definition 2.1]{B07}), then the main theorem of \cite{B07} yields $b(L,K) \leqs 5$, so we may assume $K$ is a \emph{subspace}  subgroup. Let $n$ be the dimension of the natural module for $G_0$ and note that $n \leqs 2r+1$. 

If $L = G_0$, then by combining the bounds in \cite[Theorem 3.3, Proposition 3.5]{HLM}, we deduce that $b(L,K) \leqs n+11 \leqs 2r+12$. In the general case, by arguing as in the proof of \cite[Theorem 3.1]{HLM}, using the fact that $L/G_0$ is a solvable group possessing a normal series of length at most $3$ with cyclic factors, we get $b(L,K) \leqs n+14 \leqs 2r+15$. The result now follows by appealing to Proposition \ref{p:bases}(i).

\vs

This completes the proof of Theorem \ref{t:class}.

\section{Sporadic groups: Proof of Theorems \ref{t:main0}-\ref{t:main2}}\label{s:sporadic}

In this section, we prove our main results on the depth of subgroups of almost simple sporadic groups. In particular, we establish Theorems \ref{t:main0}, \ref{t:main} and \ref{t:main2}.

\subsection{The case $G \ne \mathbb{B}$}\label{ss:nb}

As noted in Section \ref{s:intro}, the Baby Monster $\mathbb{B}$ requires special attention and we will handle it separately in Section \ref{ss:bm}. So throughout this section, we will assume $G \ne \mathbb{B}$. 

To begin with, we will focus on the simple sporadic groups and our initial goal is to prove Theorems \ref{t:main0} and \ref{t:main} when $G \ne \mathbb{B}$ is simple. To do this, it will be 
convenient to divide the possibilities for $G$ into the following three collections:
\begin{align*}
\mathcal{A} & = \{ {\rm M}_{11}, {\rm M}_{12}, {\rm M}_{22}, {\rm M}_{23}, {\rm M}_{24}, {\rm J}_{1},  {\rm J}_{2},  {\rm J}_{3}, {\rm HS}, {\rm He}, {\rm Ru}, {\rm McL}, {\rm Suz}, {\rm Co}_3, {\rm Co}_2, {\rm Co}_1, \\
& \hspace{7mm} {\rm Fi}_{22}, {\rm Fi}_{23}, {\rm O'N}, {\rm HN} \}  \\
\mathcal{B} & = \{ {\rm Ly}, {\rm Th}, \mathbb{M}\} \\
\mathcal{C} & = \{ {\rm Fi}_{24}', {\rm J}_4 \}
\end{align*}

\begin{prop}\label{p:1}
The conclusions to Theorems \ref{t:main0} and \ref{t:main} hold when $G \in \mathcal{A}$.
\end{prop}

\begin{proof}
Let $H$ be a maximal subgroup of $G$ and set $d = d_G(H)$. We begin by inspecting \cite{BOW} in order to identify the possibilities for $H$ with $b(G,H) \geqs 3$, recalling that $d = 3$ if $b(G,H) = 2$ (see Proposition \ref{p:bases}(ii)).

Working with {\sc Magma} \cite{magma}, we use the function \texttt{AutomorphismGroupSimpleGroup} to construct $G$ as a permutation group and this allows us in each case to determine the character table of $G$. Next we obtain $H$ as a subgroup of $G$ in this representation, using the \texttt{MaximalSubgroups} function, and we construct the character table of $H$. We can then compute the inner product $\la \a^G,\chi \ra$ for all $\a \in {\rm Irr}(H)$, $\chi \in {\rm Irr}(G)$, which yields the inclusion matrix $M$ corresponding to the embedding of $H$ in $G$. From here, it is straightforward to identify the minimal integer $k \geqs 3$ with $M^{(k-1)} > 0$, where $M^{(k-1)}$ is defined as in \eqref{e:Mdef}, and by appealing to Proposition \ref{p:inc}(ii) we conclude that $d= k$.

If $d = 3$, then Proposition \ref{p:graph}(vi) implies that $d_G(K) = 3$ for every nontrivial subgroup $K \leqs H$. On the other hand, if $d \geqs 4$, then we record $(d,H)$ in Table \ref{tab:max} and we repeat the above process, working with a set of representatives of the $H$-classes of maximal subgroups $K$ of $H$. In practice, if the bound $|K|^2 < |G|$ is satisfied, then we first randomly search for an element $x \in G$ with $K \cap K^x = 1$. If this search is successful after a specified number of attempts, then $b(G,K) = 2$ and thus $d_G(K) = 3$. Otherwise, we construct the character table of $K$ and we then compute $d_G(K)$ via the inclusion matrix for the embedding of $K$ in $G$.

This iterative procedure is straightforward to implement in {\sc Magma} and in all cases we find that every nontrivial $4$-maximal subgroup of $G$ has depth $3$. In  fact, almost all $3$-maximal subgroups have depth $3$ (the exceptions are recorded in Table \ref{tab:main3}, which is located in Section \ref{s:tab}). 
\end{proof}

\begin{prop}\label{p:2}
The conclusions to Theorems \ref{t:main0} and \ref{t:main} hold when $G \in \mathcal{B}$.
\end{prop}

\begin{proof}
Let $H$ be a maximal subgroup of $G$ and recall that we can read off the cases with $b(G,H) \geqs 3$ by inspecting \cite{BOW}.

First assume $G = {\rm Th}$ or $\mathbb{M}$, in which case the possibilities for $(G,H)$ are as follows:
\[
({\rm Th}, {}^3D_4(2){:}3), \; ({\rm Th}, 2^5.{\rm L}_5(2)), \; (\mathbb{M}, 2.\mathbb{B}).
\]
In each of these cases, we work with \textsf{GAP} \cite{GAP} and the \textsf{GAP} Character Table Library \cite{GAPCTL} to access the character tables of $G$ and $H$, together with the fusion map from $H$-classes to $G$-classes (we use the function \texttt{NamesOfFusionSources} to obtain the character table of $H$). We can then compute the inclusion matrix for the embedding of $H$ in $G$, and then by applying Proposition \ref{p:inc} we deduce that $d_G(H) = 3$ for the two cases with $G = {\rm Th}$ (and thus $d_G(K) = 3$ for every proper nontrivial subgroup $K$ of $G$), whereas $d_G(H) = 5$ when $G = \mathbb{M}$ and $H = 2.\mathbb{B}$. In the latter case, \cite[Theorem 3.1]{B23} gives $b(G,K) = 2$ for every proper nontrivial subgroup $K \ne H$ of $G$, which means that $d_G(K) = 3$.

For the remainder, we may assume $G = {\rm Ly}$. In this case, \cite{BOW} implies that $b(G,H) \geqs 3$ if and only if $H = G_2(5)$ or $3.{\rm McL}.2$. And using the above approach, working with \cite{GAPCTL} to compute the inclusion matrix, it is easy to check that $G_2(5)$ has depth $3$, whereas $3.{\rm McL}.2$ has depth $5$. 

Next let $K$ be a maximal subgroup of $H = 3.{\rm McL}.2$, so $K = 3.J$ for some maximal subgroup of $J$ of ${\rm McL}.2$ (note that ${\rm McL}$ is not a subgroup of $G$). By using the \textsf{GAP} function  \texttt{Maxes} with respect to $H$, we can retrieve the character table of $K$ and we can compute $\a^G = (\a^H)^G$ for all $\a \in {\rm Irr}(K)$. In this way, we deduce that either $d_G(K) = 3$, or $K = 3.{\rm McL}$ and $d_G(K) = 5$. Finally, let $L$ be a maximal subgroup of $K = 3.{\rm McL}$, so $L = 3.J$ with $J$ maximal in ${\rm McL}$. Using the \texttt{Maxes} function with respect to $K$, we get the character table of $L$ and the fusion map from $L$ to $K$. This allows us to compute $\a^G = ((\a^K)^H)^G$ for all $\a \in {\rm Irr}(L)$ and it is then easy to check that $d_G(L) = 3$ in every case.
\end{proof}

We are now in a position to complete the proofs of Theorems \ref{t:main0} and \ref{t:main}.

\begin{prop}\label{p:3}
The conclusion to Theorems \ref{t:main0} and \ref{t:main} hold when $G \in \mathcal{C}$.
\end{prop}

\begin{proof}
Here $G = {\rm Fi}_{24}'$ or ${\rm J}_4$ and we consider each possibility in turn.

\vs

\noindent \emph{Case 1. $G = {\rm Fi}_{24}'$}

\vs

Let $H$ be a maximal subgroup of $G$. From \cite{BOW}, we first note that $b(G,H) \geqs 3$ if and only if $H$ is one of the following:
\[
{\rm Fi}_{23}, \; 2.{\rm Fi}_{22}.2, \; (3 \times {\rm P\O}_{8}^{+}(3).3).2,\; 3^{1+10}.{\rm U}_5(2).2, \; \O_{10}^{-}(2),\; 3^7.\O_7(3), \;  2^{11}.{\rm M}_{24}.
\]
Working with the \textsf{GAP} Character Table Library, we can compute $d_G(H)$ via the corresponding inclusion matrix and we get $d_G(H) = 7,5,5,4,3,3,3$ in the respective cases.

We now switch to {\sc Magma} and we use the function \texttt{AutomorphismGroupSimpleGroup} to construct $A = {\rm Aut}(G) = G.2$ as a permutation group of degree $306936$. By taking the socle of $A$ we obtain $G$ and we then construct the character table of $G$ in {\sc Magma}. Let $H$ be one of the four maximal subgroups with $d_G(H)>3$:
\[
{\rm Fi}_{23}, \; 2.{\rm Fi}_{22}.2, \; (3 \times {\rm P\O}_{8}^{+}(3).3).2,\; 3^{1+10}.{\rm U}_5(2).2.
\]
Generators for $N_A(H) = H.2$ are given in the Web Atlas \cite{WebAt} in terms of a pair of  standard generators for $A$ and we obtain $H$ by intersecting with $G$. We can then use \texttt{MaximalSubgroups} to construct a set of representatives of the $H$-classes of maximal subgroups $K$ of $H$. 

\vs

\noindent \emph{Case 1(a). $H = 3^{1+10}.{\rm U}_5(2).2$.}

\vs

Here $d_G(H) = 4$ and we claim that $d_G(K) = 3$ for every maximal subgroup $K$ of $H$. If $K \ne 3^{1+10}.{\rm U}_5(2)$, then by random search it is easy to show that $K \cap K^x = 1$ for some $x \in G$, which means that $b(G,K) = 2$ and $d_G(K) = 3$. And for $K = 3^{1+10}.{\rm U}_5(2)$ we can use {\sc Magma} to construct the character table of $K$, which then allows us to compute the inclusion matrix and we deduce that $d_G(K) = 3$, as claimed.

\vs

\noindent \emph{Case 1(b). $H = (3 \times {\rm P\O}_{8}^{+}(3).3).2$.}

\vs

Here $d_G(H) = 5$ and $H$ has $12$ conjugacy classes of maximal subgroups. For a representative $K$ of $7$ of these classes, it is easy to use random search to show that $b(G,K) = 2$ and thus $d_G(K) = 3$. The remaining possibilities for $K$ are as follows:
\[
{\rm P\O}_{8}^{+}(3).S_3 \, \mbox{(3 classes)},\; 3.{\rm P\O}_{8}^{+}(3).2,\; 3 \times {\rm P\O}_{8}^{+}(3).3.
\]

First assume $K$ represents one of the three classes of maximal subgroups of $H$ isomorphic to ${\rm P\O}_{8}^{+}(3).S_3$. By constructing the character table of $K$, we deduce that $d_G(K) = 3$ for two of these subgroups, while the third has depth $5$. Let $K_1 = {\rm P\O}_{8}^{+}(3).S_3$ be the latter subgroup with depth $5$. We can construct a set of representatives of the conjugacy classes of maximal subgroups of $K_1$; there are $9$ such classes and $7$ contain subgroups $L$ with $b(G,L) = 2$, which we can verify by random search. The remaining maximal subgroups are of the form ${\rm P\O}_8^{+}(3).2$ and ${\rm P\O}_8^{+}(3).3$; by constructing the relevant character tables, we find that the former have depth $3$, while the latter have depth $5$. So let $L_1 = {\rm P\O}_8^{+}(3).3$ be a maximal subgroup of $K_1$ with depth $5$. Then $L_1$ has $8$ classes of maximal subgroups; for $7$ of these subgroups $J$, it is easy to check that $b(G,J) = 2$ by random search; the remaining maximal subgroup is $J = {\rm P\O}_{8}^{+}(3)$ and we can show that $d_G(J) = 3$ by constructing the character table of $J$ and the corresponding inclusion matrix.

Next suppose $K = 3.{\rm P\O}_{8}^{+}(3).2$. Here $d_G(K) = 4$ and $K$ has $14$ classes of maximal subgroups. For a representative $L$ of $12$ of these classes, we use random search to show that $b(G,L) = 2$. The remaining two subgroups are
isomorphic to ${\rm P\O}_{8}^{+}(3).2$ and $3.{\rm P\O}_{8}^{+}(3)$, and in both cases we get $d_G(K) = 3$ by constructing the respective character tables and inclusion matrices.
  
Finally, suppose $K = 3 \times {\rm P\O}_{8}^{+}(3).3$. Here $d_G(K) = 5$ and $K$ has $11$ conjugacy classes of maximal subgroups $L$. Using random search, it is easy to check that $b(G,L) = 2$ when $L$ is contained in $7$ of these classes; this leaves us with $3$ classes of subgroups isomorphic to ${\rm P\O}_8^{+}(3).3$ and one class of subgroups of the form $3.{\rm P\O}_8^{+}(3)$. In each case, we construct the character table and inclusion matrix for $L$. In this way, we find that if $L \cong {\rm P\O}_8^{+}(3).3$ then either $d_G(L) = 3$, or $L$ is $G$-conjugate to the subgroup denoted $L_1$ above with $d_G(L_1) = 5$. And for $L = 3.{\rm P\O}_8^{+}(3)$ we get $d_G(L) = 3$. This completes our analysis of the subgroups of $(3 \times {\rm P\O}_8^{+}(3).3).2$. 

\vs

\noindent \emph{Case 1(c). $H = 2.{\rm Fi}_{22}.2$ or ${\rm Fi}_{23}$}

\vs

First assume $H = 2.{\rm Fi}_{22}.2$. Here $d_G(H) = 5$ and $H$ has $13$ conjugacy classes of maximal subgroups. The largest one is $K = 2.{\rm Fi}_{22}$ and by constructing the character table of $K$ we deduce that  $d_G(K) = 5$. In turn, $K$ itself has $14$ classes of maximal subgroups $L$ and by random search it is easy to check that $b(G,L) = 2$ in all cases. Similarly, if $K$ is one of the remaining maximal subgroups of $H$, then $b(G,K) = 2$ by random search.

Finally, suppose $H = {\rm Fi}_{23}$. In this case $d_G(H) = 7$ and $H$ has $14$ classes of maximal subgroups. For $12$ of these subgroups $K$, we get $b(G,K) = 2$ by random search. The remaining maximal subgroups are $K = 2.{\rm Fi}_{22}$ and ${\rm P\O}_{8}^{+}(3).S_3$, both of which are $G$-conjugate to subgroups we have already handled (namely, $2.{\rm Fi}_{22} < 2.{\rm Fi}_{22}.2$ and a depth $5$ maximal subgroup ${\rm P\O}_{8}^{+}(3).S_3$ of $(3 \times {\rm P\O}_8^{+}(3).3).2$).

\vs

\noindent \emph{Case 2. $G = {\rm J}_4$}

\vs

Finally, suppose $G = {\rm J}_4$. By the main theorem of \cite{BOW}, the maximal subgroups $H$ of $G$ with $b(G,H) \geqs 3$ are as follows:
\[
2^{11}{:}{\rm M}_{24}, \; 2^{1+12}.3.{\rm M}_{22}{:}2, \; 2^{10}{:}{\rm L}_5(2).
\]
We can retrieve the character tables of $G$ and $H$ via the \textsf{GAP} Character Table Library \cite{GAPCTL} and in this way we check that $d_G(H) = 3$ when $H = 2^{10}{:}{\rm L}_5(2)$, whereas $d_G(H) = 5$ in the other two cases. 

Suppose $H =  2^{11}{:}{\rm M}_{24}$ and let $K$ be a maximal subgroup of $G$. Working with {\sc Magma}, we use \texttt{MatrixGroup("J4",1)} to construct $G$ as a subgroup of ${\rm GL}_{112}(2)$ and we can use generators given in the Web Atlas \cite{WebAt} to construct $H$ as a subgroup of $G$ in this representation. We then use the function \texttt{LMGMaximalSubgroups} to construct a representative $K$ of each conjugacy class of maximal subgroups of $H$ (this is part of the \texttt{CompositionTree} {\sc Magma} package for computing with matrix groups; see \cite{CT}). There are $11$ conjugacy classes of such subgroups; excluding the three largest maximal subgroups, we can use random search to find an element $x \in G$ such that $K \cap K^x = 1$ (as before, this means that $b(G,K) = 2$ and thus $d_G(K) = 3$). In this way, we reduce to the following possibilities for $K$:
\[
2^{11}.{\rm M}_{23}, \; 2^{11}.{\rm M}_{22}.2, \; 2^{11}.2^4.A_8.
\]

Let us assume $K$ is one of these groups. To handle these cases, we first  construct the character table of $K$ in {\sc Magma}, using the function \texttt{LMGCharacterTable}. Moving over to \textsf{GAP}, we then use the function \texttt{GAPTableOfMagmaFile} to convert the table into \textsf{GAP}-readable form and we retrieve the character tables of $G$ and $H$ (for the latter, we use the \texttt{Maxes} function). After reading in the character table of $K$, we use  \texttt{PossibleClassFusions} to determine a set of candidate fusion maps from $K$-classes to $H$-classes. Fix a character $\a \in {\rm Irr}(K)$ and fix a possible fusion map. Then we can use the function \texttt{InducedClassFunctionsByFusionMap} to compute a candidate for $\a^H$, which we can then induce to $G$ using the stored fusion map from $H$-classes to $G$-classes.  In particular, for each candidate fusion map, we can construct a corresponding candidate inclusion matrix $M$ for the embedding of $K$ in $G$, one of which must be the genuine inclusion matrix. It is easy to check that $M^{(2)}>0$ for every $M$ and this allows us to conclude that $d_G(K) = 3$ (see Proposition \ref{p:inc}).

Finally, suppose $H =  2^{1+12}.3.{\rm M}_{22}{:}2$. The analysis here is very similar to the previous case. As before, we begin by using {\sc Magma} to construct $G$ and $H$ as subgroups of ${\rm GL}_{112}(2)$, and we use \texttt{LMGMaximalSubgroups} to construct a representative $K$ of each $H$-class of maximal subgroups of $H$. There are $8$ such classes and by excluding two possibilities for $K$, namely
\[
2^{1+12}.3.{\rm M}_{22},\;  2^{1+12}.3.{\rm L}_3(4).2,
\]
we can use random search to show that $b(G,K) = 2$ and thus $d_G(K) = 3$. For the remaining two subgroups, we construct the character table of $K$ in {\sc Magma}, convert to \textsf{GAP} form and we then compute a set of candidate inclusion matrices (for the embedding of $K$ in $G$), one for each possible fusion map from $K$-classes to $H$-classes. As before, for each such matrix $M$ we find that $M^{(2)}>0$ and we deduce that $d_G(K) = 3$.
\end{proof}

We are now ready to give a proof of Theorem \ref{t:main2}.

\begin{proof}[Proof of Theorem \ref{t:main2}]
As in the statement of the theorem, let $G = G_0.2$ be an almost simple sporadic group with socle $G_0$ and let $H$ be a proper nontrivial subgroup of $G$. If $H = G_0$, then $d_G(H) = 2$ by Proposition \ref{p:socle}, so we may assume $H$ is core-free. 

First assume $G \ne {\rm Fi}_{24}$. Here we can proceed as in the proof of Proposition \ref{p:1}, working with {\sc Magma} and the permutation representation of $G$ provided by the function \texttt{AutomorphismGroupSimpleGroup}. We first calculate the character table of $G$ and we obtain a set of representatives of the $G$-classes of core-free maximal subgroups $H$ of $G$ and $G_0$ with $b(G,H) \geqs 3$. In each case, we construct the character table of $H$ and we compute $d_G(H)$ from the corresponding inclusion matrix. If $d_G(H) \geqs 4$ then $(G,H)$ is recorded in Table \ref{tab:main41} or \ref{tab:main42}, and we repeat the process with respect to a set of representatives of the classes of maximal subgroups of $H$. And we continue to descend in the subgroup lattice of $G$ until every maximal subgroup we are considering has depth $3$.

It is straightforward to automate this iterative process in {\sc Magma} and all the relevant subgroups $H$ with $d_G(H) \geqs 4$ are listed in Tables \ref{tab:main41} or \ref{tab:main42} (each subgroup is recorded up to conjugacy in $G$). In the final column of each table, we also indicate whether or not $H$ is maximal in $G$, and if it is non-maximal then we present a chain
$H = H_0 < H_1 < \cdots < H_t$, where each $H_i$ is maximal in $H_{i+1}$, and $H_t$ is a maximal subgroup of $G$. Here it is perhaps worth noting that $G = {\rm Fi}_{22}.2$ has $59$ conjugacy classes of core-free subgroups $H$ with $d_G(H) \geqs 4$, including two $4$-maximal subgroups with $d_G(H)=5$ (see Remark \ref{r:t7}).

\vs

Finally, to complete the proof we may assume $G = {\rm Fi}_{24}$. Here we can proceed as in the proof of Proposition \ref{p:3} (see Case 1). 

First let $H$ be a maximal subgroup of $G$. By inspecting \cite{BOW} we see that we only need to consider the following possibilities for $H$:
\begin{equation}\label{e:list3}
{\rm Fi}_{23} \times 2, \; (2.{\rm Fi}_{22}.2).2, \; S_3 \times {\rm P\O}_{8}^{+}(3).S_3,\; {\rm O}_{10}^{-}(2),\; 3^7.{\rm SO}_7(3), \; 3^{1+10}.({\rm U}_5(2).2 \times 2)
\end{equation}
\begin{equation}\label{e:list4}
2^{12}.{\rm M}_{24}, \; (2 \times 2^2.{\rm U}_6(2)).S_3.
\end{equation}
In each case, using \textsf{GAP}, we can access the character tables of $G$ and $H$, together with the associated fusion map, and by computing the corresponding inclusion matrix we deduce that $d_G(H) = 9,5,5,5,5,5$ for the subgroups in \eqref{e:list3}, whereas $d_G(H) = 3$ for the two subgroups in \eqref{e:list4}. So to complete the proof, it remains to determine $d_G(K)$ when $K$ is contained in one of the $6$ maximal subgroups listed in \eqref{e:list3}.

To do this, we switch to {\sc Magma} and we construct $G$ as a permutation group of degree $306936$ via the function \texttt{AutomorphismGroupSimpleGroup} and we construct the character table of $G$. Next we use generators in the Web Atlas \cite{WebAt} to obtain each maximal subgroup $H$ in \eqref{e:list3} and from here we can proceed in the usual manner, as described in Section \ref{sss:mag}, using the function \texttt{MaximalSubgroups} to descend from one layer to the next in the subgroup lattice of $H$. As before, if we are seeking to compute the depth of a given subgroup $K$, then we first check to see if $|K|^2 < |G|$; if this inequality is satisfied,  then we use random search to seek an element $x \in G$ with $K \cap K^x = 1$, which if successful implies that $d_G(K) = 3$. Otherwise, we construct the character table of $K$ and we calculate $d_G(K)$ via the inclusion matrix. 

In total, we find that $G$ has exactly $27$ conjugacy classes of subgroups $H$ with $d_G(H) \geqs 4$, and they are all recorded in Tables \ref{tab:main41} and \ref{tab:main42} (and in the final column of each table, we present a maximal chain containing $H$, as described above).
\end{proof}

\vs

This completes the proof of Theorem \ref{t:main2}.

\subsection{The Baby Monster}\label{ss:bm}

In order to complete the proof of Theorems \ref{t:main0} and \ref{t:main}, we may assume $G = \mathbb{B}$ is the Baby Monster. 

\begin{prop}\label{p:b1}
The conclusion to Theorem \ref{t:main0} holds when $G = \mathbb{B}$.
\end{prop}

\begin{proof}
Let $H$ be a maximal subgroup of $G$. By inspecting \cite{BOW,NNOW}, we see that $b(G,H) \geqs 3$ if and only if $H$ is one of the following:
\[
2.{}^2E_6(2).2, \; 2^{1+22}.{\rm Co}_2, \; 2^{9+16}.{\rm Sp}_8(2), \; {\rm Fi}_{23}, \;  {\rm Th},  \; 2^{2+10+20}.({\rm M}_{22}{:}2 \times S_3), \; (2^2 \times F_4(2)){:}2.
\]
In every case, the character tables of $G$ and $H$ are available in the \textsf{GAP} Character Table Library \cite{GAPCTL}. In addition, if $H \ne (2^2 \times F_4(2)){:}2$ then the fusion map from $H$-classes to $G$-classes is also available in \cite{GAPCTL}, which allows us to calculate the corresponding inclusion matrix. In this way, we deduce that $d_G(H) = 7,5,5,3,3,3$ for the first six subgroups listed above.

Finally, suppose $H = (2^2 \times F_4(2)){:}2$. Here the fusion map is not known, but we can use the \textsf{GAP} function \texttt{PossibleClassFusions} to compute $64$ candidate fusion maps (one of which must be the correct map). As before, we work with   \texttt{InducedClassFunctionsByFusionMap} to compute a candidate for $\a^G$ as we range over the irreducible characters $\a$ of $H$. We can then construct a candidate inclusion matrix $M$ (one for each possible fusion map) and it is easy to check that $M^{(2)}>0$ in every case. We conclude that $d_G(H) = 3$. 
\end{proof}

\vs

This completes the proof of Theorem \ref{t:main0}.

\begin{prop}\label{p:b2}
The conclusion to Theorem \ref{t:main} holds when $G = \mathbb{B}$.
\end{prop}

\begin{proof}
Let $H$ be a nontrivial and non-maximal subgroup of $G$. In view of the previous proposition (and Proposition \ref{p:graph}(vi)), we may assume $H<A$, where $A$ is one of the following maximal subgroups of $G$:
\[
2.{}^2E_6(2).2, \; 2^{1+22}.{\rm Co}_2, \; 2^{9+16}.{\rm Sp}_8(2).
\]

\vs

\noindent \emph{Case 1. $A = 2.{}^2E_6(2).2$}

\vs

We begin by assuming $H<A$, where $A = 2.{}^2E_6(2).2$. Let $Z = Z(A) = C_2$ be the centre of $A$ and note that every maximal subgroup $B$ of $A$ contains $Z$, so $B = 2.J$ for some maximal subgroup $J$ of $A/Z = {}^2E_6(2).2$.

First assume $J = {}^2E_6(2)$. Here the character table of $B= 2.J$ (as well as the fusion map from $B$-classes to $A$-classes) is available in \cite{GAPCTL} via the \texttt{NamesOfFusionSources} function. This allows us to compute the inclusion matrix for the embedding of $B$ in $G$ and we deduce that $d_G(B) = 5$, so the pair $(G,B)$ is recorded in Table \ref{tab:main2}.

Now suppose $J \ne {}^2E_6(2)$. Then by inspecting the proof of \cite[Theorem 4.1]{B23}, we see that $b(G,B) \leqs 3$ and thus $d_G(B) \leqs 5$ by Proposition \ref{p:bases}. In view of Proposition \ref{p:graph}(vi), this allows us to conclude that if $H$ is \emph{any} subgroup of $G$, then either $d_G(H) \leqs 5$, or $H = 2.{}^2E_6(2).2$ and $d_G(H) = 7$. This completes the proof of the proposition in Case 1 (we refer the reader to Remark \ref{r:2e6} for further commentary on the remaining open cases arising here).

\vs

\noindent \emph{Case 2. $A = 2^{1+22}.{\rm Co}_2$}

\vs

Now assume $A = 2^{1+22}.{\rm Co}_2$. As noted in the proof of Proposition \ref{p:b1}, we have $d_G(A) = 5$ and thus $d_G(H) \leqs 5$ for every subgroup $H$ of $A$. This is already enough to prove Theorem \ref{t:main} for $G = \mathbb{B}$ with $H < A$, and we refer the reader to Remark \ref{r:co2} for a more detailed analysis of this case.

\vs

\noindent \emph{Case 3. $A = 2^{9+16}.{\rm Sp}_8(2)$}

\vs

Recall that $d_G(A) = 5$. We claim that $d_G(K) = 3$ for every maximal subgroup $K$ of $A$, which implies that $d_G(H) = 3$ for every proper nontrivial subgroup $H$ of $A$. 

To justify the claim, we first use 
\[
\verb|OneAtlasGeneratingSet("2^(9+16).S8(2)",Dimension,180).generators|
\]
to obtain a pair of generating matrices for $A$ in \textsf{GAP}, which gives a representation of $A$ as a subgroup of ${\rm GL}_{180}(2)$. Working with these matrices, we switch to {\sc Magma} and we construct $A$. Then using  \texttt{CompositionTree} \cite{CT} and the \texttt{LMGMaximalSubgroups} function, we construct a set of representatives of the $11$ conjugacy classes of maximal subgroups of $A$. 

For each such subgroup $K$, we use \texttt{LMGCharacterTable} to construct the character table of $K$, which we convert in to \textsf{GAP}-readable form (via the function \texttt{GAPTableOfMagmaFile}). We then work with \texttt{PossibleClassFusions} from $K$-classes to $H$-classes, which allows us to compute a collection of candidate inclusion matrices $M$ corresponding to the embedding of $K$ in $G$ (one of which is the correct inclusion matrix). In each case we check that $M^{(2)}>0$ and this allows us to conclude that $d_G(K) = 3$, as claimed.
\end{proof}

\vs

This completes the proof of Theorem \ref{t:main}.

\vs

Notice that Proposition \ref{p:b2} reduces the problem of determining the depth of every subgroup $H$ of $G = \mathbb{B}$ to the situation where $H < A <G$ and $A = 2.{}^2E_6(2).2$ or $2^{1+22}.{\rm Co}_2$, with $H \ne 2.{}^2E_6(2)$ in the former case. We conclude this section by presenting two extended remarks in order to explain how the analysis of these outstanding open cases can be pushed further.

\begin{rem}\label{r:2e6}
Let $G = \mathbb{B}$ and let $H$ be a nontrivial subgroup such that $H \leqs B < A <G$, where $A = 2.{}^2E_6(2).2$ and $B = 2.J$ with $J \ne {}^2E_6(2)$ maximal in ${}^2E_6(2).2$. As noted in the proof of Proposition \ref{p:b2}, we have $b(G,H) \leqs 3$ and thus $d_G(H) \leqs 5$. Let $Z = Z(A)$.

\begin{itemize}\addtolength{\itemsep}{0.2\baselineskip}
\item[{\rm (a)}] The possibilities for $J$ are described in \cite{Craven,Wil}. First, there are $4$ conjugacy classes of maximal parabolic subgroups:
\[
P_2 = [2^{21}].{\rm U}_6(2).2,\; P_{1,6} = [2^{24}].{\rm O}_{8}^{-}(2),\; P_4 = [2^{29}].({\rm L}_3(4) \times S_3).2,
\]
\[
P_{3,5} = [2^{31}].({\rm L}_3(2) \times A_5).2,
\]
and the remaining maximal subgroups are as follows:
\[
2 \times F_4(2),\; {\rm Fi}_{22}{:}2, \; {\rm O}_{10}^{-}(2), \; S_3 \times {\rm U}_6(2).2, \; S_3 \times \O_8^{+}(2).S_3, \; {\rm U}_3(8).6,
\]
\[
({\rm L}_3(2) \times {\rm L}_3(4).2).2, \; 3^{1+6}.2^{3+6}.3^2.2^2,\; {\rm O}_7(3), \; {\rm U}_3(2).2 \times {\rm U}_3(3).2.
\] 

\item[(b)] First assume $J = {\rm U}_3(8).6$. Let $\chi$ be the permutation character for the action of $G$ on $G/B$ and let $\a$ be the trivial character of $B/Z = J$. Then 
\[
\chi = {\rm Ind}^G_A{\rm Infl}^A_{A/Z}{\rm Ind}^{A/Z}_{B/Z} \a,
\]
where ${\rm Ind}^G_A \psi$ is the character of $G$ induced from the character $\psi$ of $A$, and ${\rm Infl}^A_{A/Z} \eta$ denotes the inflation of a character $\eta$ of $A/Z$ to $A$. Using \texttt{NamesOfFusionSources} (with respect to $A/Z = {}^2E_6(2).2$) we can retrieve the character table and fusion map in \textsf{GAP} for the embedding of $J$ in ${}^2E_6(2).2$. So using \textsf{GAP} we can compute $\chi$ and this allows us to show that $Q(G,2)<1$ with respect to the action of $G$ on $G/B$ (see \eqref{e:QGt} in Section \ref{sss:bases}). Then Proposition \ref{p:prob}(ii) yields $d_G(B) = 3$ and thus $d_G(H) = 3$ in this situation.

\item[(c)] For $J = {\rm O}_{10}^{-}(2)$, ${\rm O}_7(3)$ or ${\rm Fi}_{22}{:}2$, we can directly  access the character table of $J$ and by using \texttt{PossibleClassFusions} (with respect to the embedding of $J$ in ${}^2E_6(2).2$) we can compute a set of candidates for the permutation character corresponding to the action of $G$ on $G/B$. Then by arguing as in (b), we are able to deduce that $Q(G,2)<1$ and hence $d_G(H) = 3$. And the cases where $J$ is one of the following
\[
S_3 \times {\rm U}_6(2).2, \; S_3 \times \O_8^{+}(2).S_3,\; {\rm U}_3(2).2 \times {\rm U}_3(3).2
\]
can be handled in the same way, with the same conclusion. 

\item[(d)] We thank Thomas Breuer for calculating the character tables of $J = ({\rm L}_3(2) \times {\rm L}_3(4).2).2$ and $3^{1+6}.2^{3+6}.3^2.2^2$, which we can use to deduce that $Q(G,2)<1$ by arguing as above. So once again, $d_G(H) = 3$ in these cases.

\item[(e)] For $J = 2 \times F_4(2)$ we have $B = 2^2 \times F_4(2)$ and we can work directly with the character table of $B$ and a set of possible fusion maps with respect to the embedding of $B$ in $A$. In this way, we can determine a set of candidate inclusion matrices corresponding to $B<G$ and this allows us to deduce that $d_G(B) = 3$ and thus $d_G(H) = 3$.
\end{itemize}

To summarize, we have shown that $d_G(H) = 3$ for every proper nontrivial subgroup $H$ of $A = 2.{}^2E_6(2).2$, with the possible exception of subgroups $H \leqs 2.J < A$, where $J$ is a maximal parabolic subgroup of ${}^2E_6(2).2$. Here $d_G(H) \leqs 5$, but it remains an open problem to determine the exact depth of $H$.
\end{rem}

\begin{rem}\label{r:co2}
Now assume $G = \mathbb{B}$ and $H$ is a nontrivial subgroup such that $H \leqs B < A <G$, where $A = 2^{1+22}.{\rm Co}_2$ and $B$ is a maximal subgroup of $A$. We have $B = N.J$, with $N = 2^{1+22}$ and $J$ maximal in ${\rm Co}_2$. There are $11$ conjugacy classes of such subgroups in $A$. 

\begin{itemize}\addtolength{\itemsep}{0.2\baselineskip}
\item[{\rm (a)}] First assume that $J$ is one of the following:
\[
(2^4 \times 2^{1+6}).A_8, \; {\rm U}_4(3){:}D_8, \; 2^{4+10}.(S_5 \times S_3), \; {\rm M}_{23},\; 3^{1+4}.2^{1+4}.S_5,\; 5^{1+2}{:}4S_4.
\]
We claim that $b(G,B) = 2$ and thus $d_G(H) = 3$. To see this, we proceed as follows. 

Write $B = N.J$, where $N = 2^{1+22}$, and let $\chi$ be the permutation character corresponding to the action of $G$ on $G/B$. Define $Q(G,2)$ as in \eqref{e:QGt}. Let $\a$ be the trivial character of $B/N = J$ and observe that 
\begin{equation}\label{e:ind}
\chi = {\rm Ind}^G_H {\rm Ind}^H_B {\rm Infl}^B_{B/N} \a = {\rm Ind}^G_H{\rm Infl}^H_{H/N}{\rm Ind}^{H/N}_{B/N} \a.
\end{equation}
We can retrieve the character tables of $G$, $H$, $H/N$ and $B/N$ from the \textsf{GAP} Character Table Library \cite{GAPCTL}, together with all of the associated fusion maps on conjugacy classes. So this allows us to use \textsf{GAP} to compute $\chi$ and hence $Q(G,2)$. In this way, we deduce that $Q(G,2)<1$ and thus $b(G,B) = 2$ as claimed.

\item[{\rm (b)}] Next suppose $J = {\rm U}_6(2).2$. We claim that $d_G(B) = 5$. To see this, first recall that $b(G,A) = 3$, hence $b(G,B) \leqs 3$ and thus $d_G(B) \leqs 5$. To establish equality, let $\a_1,\a_2 \in {\rm Irr}(J)$ be the two linear characters of $J$, where $\a_1$ is the trivial character. Then as in (a) we can compute
\[
\chi_i = {\rm Ind}^G_H{\rm Infl}^H_{H/N}{\rm Ind}^{H/N}_{B/N} \a_i
\]
for $i=1,2$. Here $\chi_i = \b_i^G$ for distinct linear characters $\b_i$ of $B$, one of which is the trivial character, and using \textsf{GAP} we compute $\la \b_1^G, \b_2^G \ra = 0$. Therefore, Proposition \ref{p:graph}(v) implies that $d_G(B) \geqs 5$ and so we conclude that $d_G(B) = 5$, as claimed.

Next let $C$ be a maximal subgroup of $B$. To begin with, we assume $C$ contains $N$, which means that $C = N.K$ for some maximal subgroup $K$ of $J = {\rm U}_6(2).2$. There are $11$ distinct $B$-classes of such subgroups.

Suppose $K \ne {\rm U}_6(2)$ and let $\a$ be the trivial character of $K = C/N$. Then in each case, we compute the character table of $K$ in {\sc Magma}, working with a permutation representation of $J$ of degree $693$. We then convert the table to \textsf{GAP}-readable form and using \textsf{GAP} we determine the set of possible class fusion maps from $K$-classes to $J$-classes, recalling that the character table of $J = B/N$ is available in \cite{GAPCTL}. For each possible fusion map, we calculate the corresponding induced character $\psi = \a^{B/N}$ and we then use the expression in \eqref{e:ind} to compute a candidate permutation character for the action of $G$ on $G/C$. By calculating with each candidate permutation character in turn, we deduce that $Q(G,2)< 1$, which implies that $b(G,C) = 2$ and $d_G(C) = 3$.

Now assume $C = N.K$ with $K = {\rm U}_6(2)$ and let $D$ be a maximal subgroup of $C$. First assume $D$ contains $N$, so $D = N.J$ for some maximal subgroup $J$ of $K = {\rm U}_6(2)$. The character table of $J$ is available in \cite{GAPCTL} (via the \texttt{Maxes} function), together with the fusion map from $J$-classes to $K$-classes. Then by arguing as above, we can calculate the permutation character for the action of $G$ on $G/D$ and we check that $Q(G,2) < 1$, which gives $d_G(D) = 3$. Now suppose $D$ does not contain $N$, so $D = (D \cap N).K$. If we write $Z = Z(A)$, then $V = N/Z$ is a $22$-dimensional irreducible module for ${\rm Co}_2$ over $\mathbb{F}_{2}$ and we find that $V$ is uniserial as a module for $K = {\rm U}_6(2)$, with a $21$-dimensional maximal submodule. Then the maximality of $D$ in $C$ implies that $D$ is of the form $2^{1+21}.{\rm U}_6(2)$, and by arguing as above we deduce that every maximal subgroup of $D$ has depth $3$.

To continue the analysis of the case $J = {\rm U}_6(2).2$, we also need to consider the case where $H$ is contained in a maximal subgroup $C = (C \cap N).J$ of $B$. As in the previous paragraph, since the action of $J = {\rm U}_6(2).2$ on $V = N/Z$ is uniserial with a maximal submodule of dimension $21$, it follows that $C$ is of the form $2^{1+21}.{\rm U}_6(2).2$. In addition, our earlier work shows that every maximal subgroup of $C$ of the form $2^{1+21}.K$ with $K \ne {\rm U}_6(2)$ has depth $3$.

So to summarize, for $H \leqs B = 2^{1+22}.{\rm U}_6(2).2$, it just remains to determine $d_G(H)$ when $H$ is one of the following: 
\[
2^{1+22}.{\rm U}_6(2),\; 2^{1+21}.{\rm U}_6(2),\; 2^{1+21}.{\rm U}_6(2).2.
\]
Since $d_G(B) = 5$, it follows that $d_G(H) \in \{3,4,5\}$ but the character tables of these subgroups have not been computed and it remains an open problem to determine their exact depth.

\item[(c)] The analysis of the case $J = 2^{10}.{\rm M}_{22}.2$ is very similar to the previous one. Let $C$ be a maximal subgroup of $B$, and first assume that $C$ contains $N$, so $C = N.K$ with $K$ maximal in $J = 2^{10}.{\rm M}_{22}.2$. Since ${\rm M}_{22}.2$ acts irreducibly on the elementary abelian normal subgroup $2^{10}$, it follows that there are $8$ possibilities for $K$ up to conjugacy; either $K = {\rm M}_{22}.2$ or $K = 2^{10}.L$ with $L$ maximal in ${\rm M}_{22}.2$. In each case, we can argue as before, constructing character tables in {\sc Magma}, converting to \textsf{GAP}, taking possible class fusions and computing the corresponding permutation characters; in this way, we deduce that $Q(G,2)<1$ unless $C = 2^{1+22}.(2^{10}.{\rm M}_{22})$. We can then repeat the analysis to show that every maximal subgroup of $C$ of the form $N.L$ with $L$ maximal in $2^{10}.{\rm M}_{22}$ has depth $3$. And finally, noting that both $2^{10}.{\rm M}_{22}.2$ and $2^{10}.{\rm M}_{22}$ are uniserial on the $\mathbb{F}_2J$-module $V = N/Z$, with a $21$-dimensional maximal submodule, we conclude that the remaining open cases $H$ with $H \leqs B$ are the following:
\[
2^{1+22}.2^{10}.{\rm M}_{22}.2,\; 2^{1+22}.2^{10}.{\rm M}_{22},\; 2^{1+21}.2^{10}.{\rm M}_{22}.2,\; 2^{1+21}.2^{10}.{\rm M}_{22}.
\]
Since $d_G(A) = 5$, we have $d_G(H) \in \{3,4,5\}$ in each case, but we have not been able to determine the exact depth of $H$.

\item[(d)] Finally, let us assume $H \leqs B = N.J$ with $J = {\rm McL}$, $2^{1+8}.{\rm Sp}_6(2)$ or ${\rm HS}.2$. As above, let $V$ be the irreducible ${\rm Co}_2$-module $N/Z$.

Suppose $J = {\rm McL}$. Here the extension $N.J$ is non-split and the restriction of $V$ to $J$ is irreducible, whence every maximal subgroup of $B$ is of the form $C = N.K$ with $K$  maximal in $J$. Then by arguing as above, we can show that $b(G,C) = 2$ and so in this case it just remains to determine $d_G(H)$ when $H = B$.

Next assume $J  = 2^{1+8}.{\rm Sp}_6(2)$. Here we get $b(G,C) = 2$ for every maximal subgroup $C$ of the form $2^{1+22}.K$ with $K<J$ maximal. In addition, we note that the restriction of $V$ to $J$ is uniserial, with a maximal submodule of dimension $16$, and this observation allows us to conclude that either $d_G(H) = 3$, or $H=B$, or $H$ is a subgroup of $B$ of the form $2^{1+16}.J$. 

Similarly, for $J = {\rm HS}.2$ we can reduce the problem to the case where $H$ is either equal to $B$, or it is a subgroup of the form $H = 2^{1+21}.J$ (in particular, by calculating the permutation character for the action of $G$ on $G/C$, we can show that $b(G,C) = 2$ when $C$ is the index-two subgroup $2^{1+22}.{\rm HS}$ of $B$).
\end{itemize}

To summarize, we conclude that $d_G(H) = 3$ for every proper nontrivial subgroup $H$ of $A = 2^{1+22}.{\rm Co}_2$, with the possible exception of the following subgroups (with $N = 2^{1+22}$):
\begin{itemize}\addtolength{\itemsep}{0.2\baselineskip}
\item[{\rm (i)}] $N.{\rm U}_6(2)$, $2^{1+21}.{\rm U}_6(2)$, $2^{1+21}.{\rm U}_6(2).2$,
\item[{\rm (ii)}] $N.2^{10}.{\rm M}_{22}.2$, $N.2^{10}.{\rm M}_{22}$, $2^{1+21}.2^{10}.{\rm M}_{22}.2$, $2^{1+21}.2^{10}.{\rm M}_{22}$,
\item[{\rm (iii)}] $N.{\rm McL}$, $N.2^{1+8}.{\rm Sp}_6(2)$, $N.{\rm HS}.2$, $2^{1+16}.2^{1+8}.{\rm Sp}_6(2)$, $2^{1+21}.{\rm HS}.2$.
\end{itemize}
In each of these cases, we know that $d_G(H) \in \{3,4,5\}$, but it remains an open problem to determine the precise depth of $H$.
\end{rem}

\begin{rem}
In order to illustrate some of the difficulties with the remaining open cases arising in Remark \ref{r:co2}, let us consider the situation where  
$$H = N.{\rm U}_6(2) < B = N.{\rm U}_6(2).2 < A = N.{\rm Co}_2 < G = \mathbb{B}$$ 
and $N = 2^{1+22}$. As explained in part (b) of the previous remark, we know that  $d_G(B) = 5$. In addition, $|H|^2 > |G|$ and thus $b(G,H) = 3$ (in particular, we cannot try to force $d_G(H) = 3$ via Proposition \ref{p:prob}(ii) since we have $Q(G,2)>1$). If we write $Z = Z(A)$, then $A/Z = 2^{22}.{\rm Co}_2$ and we can use {\sc Magma} to calculate the character tables of $A/Z$ and $H/Z = 2^{22}.{\rm U}_6(2)$, working with a permutation representation of $A/Z$ of degree $4600$ (see \cite[p.20]{BMW}). In this way, we are able to check that $\la \a^G, \b^G \ra \ne 0$ for every pair of characters $\a,\b \in {\rm Irr}(H)$ that are inflated from $H/Z$. This is consistent with $d_G(H) = 3$, but further work is required in order to resolve this case. 
\end{rem}

\section{Proof of Theorems \ref{t:delta}-\ref{t:solv}}\label{s:solv}

In this section we give a proof of Theorem \ref{t:delta}, which classifies the finite simple groups $G$ with the property that $d_G(H) = 3$ for every proper nontrivial subgroup $H$. In addition,  we also study the depth of nilpotent and solvable subgroups of almost simple groups, which allows us to prove Theorems \ref{t:nilp} and \ref{t:solv}.

\subsection{Proof of Theorem \ref{t:delta}}

Let $G$ be a finite simple group and let $\delta(G)$ be the number of conjugacy classes of proper subgroups $H$ of $G$ with $d_G(H) \geqs 4$. Note that $\delta(G) = 0$ if and only if $d_G(H) = 3$ for every maximal subgroup $H$ of $G$. 

If $G$ is a sporadic group, then by inspecting Table \ref{tab:max} we see that $\delta(G) = 0$ if and only if $G = {\rm J}_1$, ${\rm J}_3$, ${\rm J}_4$ or ${\rm Th}$. And for $G = A_n$ and $H = A_{n-1}$, we recall that \cite[Proposition A.5]{BKK} gives $d_G(H) = 2n - 2\lceil \sqrt{n-1} \rceil - 1 \geqs 5$. For the remainder, we may assume $G$ is a simple group of Lie type over $\mathbb{F}_q$, where $q =p^f$ with $p$ a prime. We thank Gunter Malle for helpful discussions concerning the following argument.

It is easy to check that $\delta(G) > 0$ when $G$ is one of the following special cases:
\[
{\rm Sp}_4(2)' \cong A_6,\; G_2(2)' \cong {\rm U}_3(3), \; {}^2G_2(3)' \cong {\rm L}_2(8), \; {}^2F_4(2)'.
\]
So for the remainder, we may write $G = S/Z$, where $S$ is the simply connected version of $G$ and $Z = Z(S)$ (so for example, if $G = {\rm L}_n(q)$ then $S = {\rm SL}_n(q)$).

Let $B = N_S(U) = UT$ be a Borel subgroup of $S$, where $U$ is a Sylow $p$-subgroup of $S$ and $T$ is a torus. Set $H = B/Z$. For now, let us assume $T/Z$ is nontrivial, which means that we are excluding the following possibilities for $G$:
\begin{equation}\label{e:list2}
{\rm L}_n(2) \, \mbox{$(n \geqs 3)$}, \; {\rm Sp}_n(2) \, \mbox{$(n \geqs 6)$}, \; \Omega^{+}_n(2) \, \mbox{$(n \geqs 8)$}, \; F_4(2), \; E_6(2), \; E_7(2), \; E_8(2).
\end{equation}

Take two linear characters $\a,\b \in {\rm Irr}(T)$, where $\a$ is trivial and $\b$ is inflated from a nontrivial character of $T/Z$. By inflation, we may view both $\a$ and $\b$ as irreducible characters of $B$. Then $(T,\a)$ is a \emph{cuspidal pair} for $S$ (see \cite[Definition 3.1.14]{GM}) and the corresponding \emph{Harish-Chandra series}, denoted ${\rm Irr}(S,(T,\a))$, comprises the irreducible constituents of the induced character $\a^S = {\rm Ind}_B^{S}({\rm Infl}_T^B \a)$. Similarly, $(T,\b)$ is a cuspidal pair for $S$, which is not $S$-conjugate to $(T,\a)$, so \cite[Corollary 3.1.17]{GM} implies that ${\rm Irr}(S,(T,\a))$ and ${\rm Irr}(S,(T,\b))$ are disjoint. In other words, $\la \a^S,\b^S\ra = 0$. And since $Z$ is contained in the kernels of $\a$ and $\b$, we may view them as irreducible characters of $H = B/Z$, in which case $\la \a^G, \b^G \ra = 0$. Therefore, Proposition \ref{p:graph}(v) implies that $d_G(H) \geqs 5$ and thus $\delta(G)>0$.

It remains for us to consider the groups in \eqref{e:list2}. First assume $G$ is a classical group and let $H = QL$ be a maximal parabolic subgroup of $G$ of type $P_1$, which means that $H$ is the stabilizer in $G$ of a $1$-dimensional totally singular subspace of the natural module for $G$. Here $Q$ is the unipotent radical of $H$ and $L$ is a Levi subgroup. For $G = {\rm L}_n(2), {\rm Sp}_n(2), \Omega^{+}_n(2)$ we have $L = {\rm L}_{n-1}(2),{\rm Sp}_{n-2}(2),\O_{n-2}^{+}(2)$, respectively, and we take $\a,\b \in {\rm Irr}(H)$, where $\a$ is trivial and $\b = {\rm Infl}_L^H \textsf{St}$ is the inflation of the Steinberg character of $L$. Then 
\[
(\a^G)(1) = \left\{\begin{array}{ll}
(2^{n/2-1}+1)(2^{n/2}-1) & \mbox{if $G = \O_n^{+}(2)$} \\
2^n-1 & \mbox{otherwise}
\end{array}\right. 
\]
and
\[
\b(1) = |L|_2 = \left\{\begin{array}{ll}
2^{(n-1)(n-2)/2} & \mbox{if $G = {\rm L}_n(2)$} \\
2^{(n-2)^2/4} & \mbox{if $G = {\rm Sp}_n(2)$} \\
2^{(n-2)(n-4)/4} & \mbox{if $G = \O_n^{+}(2)$.} 
\end{array}\right.
\]
Clearly, if $(\a^G)(1) < \b(1)$ then $\la \a^G, \b^G \ra = \la (\a^G)_H, \b \ra = 0$ and thus $d_G(H) \geqs 5$ as before. It is routine to check that this inequality holds unless $G = {\rm L}_3(2)$, ${\rm L}_4(2)$, ${\rm Sp}_6(2)$ or $\O_8^{+}(2)$, and in each of these cases it is easy to verify the bound $d_G(H) \geqs 5$ directly, with the aid of \textsf{GAP} or {\sc Magma}.

Similarly, if $G = E_{r}(2)$ with $r \in \{6,7,8\}$ then we take $H = QL$ to be a maximal parabolic subgroup of type $P_r$, which means that the respective Levi subgroups are $L = \O_{10}^{+}(2)$, $E_6(2)$ and $E_7(2)$. We can now repeat the above argument for classical groups, taking the trivial character $\a$ of $H$ and the inflation $\b$ of the Steinberg character of $L$, noting that in every case we have $(\a^G)(1) < \b(1) = |L|_2$. Finally, if $G = F_4(2)$ then the character tables of $G$ and every maximal subgroup of $G$ are available in \cite{GAPCTL} and by computing the corresponding inclusion matrix it is easy to check that $d_G(H) = 7$ when $H$ is a maximal parabolic subgroup of type $P_1$.

\vs

This completes the proof of Theorem \ref{t:delta}.

\subsection{Proof of Theorem \ref{t:nilp}}

Let $G$ be an almost simple group and let $H$ be a nilpotent subgroup of $G$. As a special case of the main theorem of \cite{Zen21}, we have $b(G,H) \leqs 3$ and thus $d_G(H) \leqs 5$ by Proposition \ref{p:bases}(i). Moreover, if $G = S_8$ and $H$ is a Sylow $2$-subgroup of $G$, 
then it is easy to check that $b(G,H) = 3$ and $d_G(H) = 5$. This establishes part (i). Part (ii) follows immediately by combining the main theorems on $b(G,H)$ in \cite{Zenkov14,Zen20}, and (iii) follows from \cite[Corollary B]{BH1}. 

\vs

This completes the proof of Theorem \ref{t:nilp}.

\subsection{Proof of Theorem \ref{t:solv}}

Let $G$ be an almost simple group with socle $G_0$ and let $H$ be a solvable maximal subgroup of $G$. By the main theorem of \cite{B21}, we have $b(G,H) \leqs 5$ and thus $d_G(H) \leqs 9$. Of course, if $b(G,H) \leqs 4$ then $d_G(H) \leqs 7$, so it just remains for us to consider the pairs $(G,H)$ with $b(G,H) = 5$. By \cite[Theorem 2(i)]{B21}, one of the following holds:
\begin{itemize}\addtolength{\itemsep}{0.2\baselineskip}
\item[(a)] $G = S_8$ and $H = S_4 \wr S_2$;
\item[(b)] $G_0 = {\rm L}_4(3)$ and $H = P_2$;
\item[(c)] $G_0 = {\rm U}_5(2)$ and $H = P_1$.
\end{itemize}
Note that in (b) and (c), $H$ is a maximal parabolic subgroup of $G$ (more precisely, $P_m$ denotes the stabilizer in $G$ of a totally isotropic $m$-dimensional subspace of the natural module for $G_0$).

In (a), we can easily construct $G$ and $H$ in {\sc Magma} and by computing the character tables we can determine the corresponding inclusion matrix. In this way, we deduce that $d_G(H) = 9$. 

Cases (b) and (c) are similar. In (b), there are $5$ possibilities for $G$ (up to conjugacy in ${\rm Aut}(G)$) and we note that $H = N_G(H_0)$, where $H_0 = H \cap G_0$ and $|H_0| = 46656$. We can construct $G$ and $H$ in {\sc Magma}, using the functions \texttt{AutomorphismGroupSimpleGroup} and \texttt{MaximalSubgroups}, and by computing the corresponding character tables and inclusion matrices, it is easy to check that 
\[
d_G(H) = \left\{\begin{array}{ll}
9 & \mbox{if $G = {\rm L}_4(3).2_2$ or ${\rm Aut}({\rm L}_4(3)) = {\rm L}_4(3).2^2$} \\
7 & \mbox{otherwise,}
\end{array}\right.
\]
where ${\rm L}_4(3).2_2$ can be identified among the almost simple groups of the form ${\rm L}_4(3).2$ by the fact that it contains involutory graph automorphisms $x$ with $C_{G_0}(x) = {\rm PGSp}_4(3)$. 

Finally, in case (c) we have $H = N_G(H_0)$, where $H_0 = H \cap G_0$ has order $82944$,  and using {\sc Magma} once again we get
\[
d_G(H) = \left\{\begin{array}{ll}
8 & \mbox{if $G = {\rm U}_5(2)$} \\
7 & \mbox{otherwise}
\end{array}\right.
\]
and the result follows.

\vs

This completes the proof of Theorem \ref{t:solv}.

\vs

To conclude our study of the depth of solvable maximal subgroups, we present the following more refined version of Theorem \ref{t:solv}. 

\begin{thm}\label{t:solv2}
Let $G$ be a finite almost simple group with socle $G_0$, let $H$ be a solvable maximal subgroup of $G$ and set $d = d_G(H)$. Then one of the following holds:
\begin{itemize}\addtolength{\itemsep}{0.2\baselineskip}
\item[{\rm (i)}] $d \leqs 5$.
\item[{\rm (ii)}] $G_0 = {\rm L}_2(q)$, $H = P_1$ and $G \not\in \{ G_0, {\rm PGL}_2(q), \la G_0, \delta\phi^{f/2} \ra\}$ and $d \leqs 7$.
\item[{\rm (iii)}] $(G,H,d)$ is one of the cases in Table \ref{tab:solv} and we have $6 \leqs d \leqs 9$.
\end{itemize}
\end{thm}

\begin{rem}\label{r:solv2}
Before giving the proof, let us record a few comments on the statement of Theorem \ref{t:solv2}.
\begin{itemize}\addtolength{\itemsep}{0.2\baselineskip}
\item[{\rm (a)}] In part (ii), $H = P_1$ is the stabilizer in $G$ of a $1$-dimensional subspace of the natural module for $G_0 = {\rm L}_2(q)$ (in other words, $H$ is a Borel subgroup of $G$, which coincides with the normalizer in $G$ of a Sylow $p$-subgroup of $G_0$, where $q = p^f$ and $p$ is a prime), $\delta$ is a diagonal automorphism of $G_0$ and $\phi$ is a field automorphism of order $f$. Note that $\delta = 1$ if $p = 2$.

\item[{\rm (b)}] In Table \ref{tab:solv}, each subgroup $H$ is recorded up to conjugacy in $G$ (note that $G = S_6$ has two classes of maximal subgroups $H$ with $H \cong S_4 \times S_2$, and as indicated in the table, we have $d_G(H) = 7$ in both cases). As before, if $G$ is a classical group, then we write $P_m$ for the stabilizer in $G$ of an $m$-dimensional totally singular subspace of the natural module $V$ for $G_0$. And for $G_0 = {\rm L}_n(q)$ we use $P_{1,n-1}$ for the stabilizer in $G$ of a flag $0 < U < W < V$, where $\dim U = 1$ and $\dim W = n-1$ (such a subgroup is maximal when $G \not\leqs {\rm P\Gamma L}_n(q)$). In addition, we adopt the standard Atlas \cite{Atlas} notation for classical groups of the form ${\rm L}_4(3).2$ and ${\rm U}_4(3).2$, which means that ${\rm L}_4(3).2_2$ contains involutory graph automorphisms $x$ with $C_{G_0}(x) = {\rm PGSp}_4(3)$, and ${\rm U}_4(3).2_1$ contains graph automorphisms $x$ with $C_{G_0}(x) = {\rm PO}_{4}^{-}(3)$.

\item[{\rm (c)}] Consider the cases in Table \ref{tab:solv} with $G_0 = {\rm P\O}_8^{+}(3)$ and $H = P_2$. Here $d_G(H) = 7$ if and only if $G = G_0.A$, where 
$A \in \{C_3,S_3, D_8, A_4, S_4\}$, or $A = C_2 \times C_2$ and $G = {\rm Inndiag}(G_0)$, or $A = C_2$ and $G$ contains a graph automorphism $x$ with $C_{G_0}(x) = \O_7(3)$.

\item[{\rm (d)}] Recall that if $G$ is a sporadic group, then the main theorem of \cite{B23} implies that $b(G,H) \leqs 3$ for \emph{every} solvable subgroup $H$ of $G$, in which case $d_G(H) \leqs 5$. And for $G \ne \mathbb{B}$, we can apply Theorems \ref{t:main0} and \ref{t:main} in order to read off the pairs $(G,H)$ with $d_G(H) \in \{4,5\}$.

\item[{\rm (e)}]  Suppose $G_0 = {\rm L}_2(q)$ and $H = P_1$ is a Borel subgroup. Then \cite[Proposition 4.1]{B21} gives $b(G,H) \in \{3,4\}$, with $b(G,H) = 4$ if and only if $G$ satisfies the conditions in part (ii) of Theorem \ref{t:solv2}. In the latter situation, it remains an open problem to determine the exact depth $d_G(H)$. Here it is worth noting that if $G = G_0$, then $b(G,H) = 3$ and \cite[Proposition 2.16]{F1} gives $d_G(H) = 5$. And with the aid of {\sc Magma}, we have checked that if $b(G,H) = 4$ and $q \leqs 17^2$, then $d_G(H) = 7$ in every case. So there is some evidence to suggest that $d_G(H) = 7$ if and only if $b(G,H) = 4$.
\end{itemize}
\end{rem}

\begin{proof}[Proof of Theorem \ref{t:solv2}]
By Proposition \ref{p:bases}(i), we know that (i) holds if $b(G,H) \leqs 3$, and we have already handled all the cases with $b(G,H) = 5$ in the proof of Theorem \ref{t:solv}. Therefore, we may assume $b(G,H) = 4$ and we can read off the possibilities for $(G,H)$ by inspecting Tables 4--7 in \cite{B21}. If $G_0 = {\rm L}_2(q)$ and $H = P_1$, then $b(G,H) = 4$ if and only if 
\[
G \not\in \{ G_0, {\rm PGL}_2(q), \la G_0, \delta\phi^{f/2} \ra\}
\]
as in part (ii) of the theorem, so we can exclude this case. 

We are left with a finite list of pairs $(G,H)$ to consider and in each case we can use {\sc Magma} to compute $d_G(H)$, following the procedure outlined in Section \ref{sss:mag}. As usual, we first use the functions \texttt{AutomorphismGroupSimpleGroup} and \texttt{MaximalSubgroups} to construct $G$ and $H$ with respect to a suitable permutation representation of $G$, and we then calculate the character tables of $G$ and $H$. This then allows us to compute the corresponding inclusion matrix and from this we quickly determine $d_G(H)$. The reader can check that every subgroup with $d_G(H) \geqs 6$ has been listed in Table \ref{tab:solv}, up to conjugacy in $G$.
\end{proof}

{\scriptsize
\begin{table}
\[
\begin{array}{cll} \hline
d & G & H  \\ \hline
9 & {\rm L}_4(3).2^2,\, {\rm L}_4(3).2_2 & P_2  \\
& &  \\
8 & {\rm U}_5(2) & P_1  \\
& &  \\
7 & S_8 & S_4 \wr S_2  \\
& A_6.2^2 & {\rm AGL}_1(9).2 \\
& S_6 & S_4 \times S_2 \, \mbox{(two)}, \, S_3 \wr S_2  \\
& S_5 & S_4  \\
& {\rm P\O}_8^{+}(3).A \mbox{ (see Remark \ref{r:solv2}(c))} & P_2 \\
& \O_8^{+}(2).A \ne \O_8^{+}(2) & P_2 \\ 
& {\rm SO}_7(3) & P_2 \\ 
& {\rm U}_5(2).2 & P_1  \\
& {\rm L}_4(3).2^2 & P_{1,3}  \\
& {\rm L}_4(3), \, {\rm L}_4(3).2_1, \, {\rm L}_4(3).2_3 & P_2  \\
& {\rm U}_4(3).2_1 & P_1  \\
& {\rm U}_4(2).2 & P_1, \, {\rm GU}_3(2).2, \, 3^3.S_4.2  \\
& {\rm U}_4(2) & P_1,\, {\rm GU}_3(2), \, 3^3.S_4  \\
& {\rm L}_3(4).D_{12} & P_{1,2}  \\
& {\rm L}_3(3) & P_1, \, P_2  \\
& &  \\
6 & A_8 & (S_4 \wr S_2) \cap G  \\
\hline
\end{array}
\]
\caption{The pairs $(G,H)$ in part (iii) of Theorem \ref{t:solv2}}
\label{tab:solv}
\end{table}
}

\section{The tables}\label{s:tab}

In this final section, we present the tables referred to in the statements of Theorems \ref{t:sn}, \ref{t:main0}, \ref{t:main} and \ref{t:main2}.

{\scriptsize
\begin{table}
\[
\begin{array}{clll} \hline
d & G & H &  \\ \hline
9 & S_6 & S_5 \, \mbox{(primitive)} & \mbox{maximal} \\
& & & \\
7 & S_8 & {\rm AGL}_3(2) & H < A_8 \\
& A_8 & {\rm AGL}_3(2) \, \mbox{(two)} & \mbox{maximal} \\
& A_6.2^2 & S_5 & H < S_6 \\
& & & \\
5 & S_{12} & {\rm M}_{12} & H < A_{12} \\ 
& A_{12} & {\rm M}_{12} \, \mbox{(two)} & \mbox{maximal} \\
& S_9 & {\rm L}_2(8).3 & H < A_9 \\
& A_9 & {\rm L}_2(8).3 \, \mbox{(two)} & \mbox{maximal} \\
& S_8 & {\rm PGL}_2(7) & \mbox{maximal} \\
& & {\rm L}_2(7) & H < {\rm PGL}_2(7) \\
& & 2^3{:}(7{:}3) & H < {\rm AGL}_3(2) < A_8 \\
& A_8 & {\rm GL}_3(2) & H < {\rm AGL}_3(2) \\
& & 2^3{:}(7{:}3) \, \mbox{(two)} & H < {\rm AGL}_3(2) \\
& S_7 & {\rm L}_2(7) & H < A_7 \\
& A_7 & {\rm L}_2(7) \, \mbox{(two)} & \mbox{maximal} \\
& A_6.2^2 & A_5 & H < S_5 < S_6 \\
& S_6 & A_5 & H < A_6 \\
& A_6.2 = {\rm M}_{10} & A_5 & H < A_6 \\
& A_6 & A_5 & \mbox{maximal} \\
& S_5 & 5{:}4 & \mbox{maximal} \\
& & D_{10} & H < 5{:}4 \\
& A_5 & D_{10} & \mbox{maximal} \\
& & & \\ 
4 & S_{10} & A_6.2^2 & \mbox{maximal} \\
\hline
\end{array}
\]
\caption{The pairs $(G,H)$ in part (ii) of Theorem \ref{t:sn}}
\label{tab:sn}
\end{table}
}

\begin{rem}\label{r:tabs}
Let us comment on the set-up we adopt in Tables \ref{tab:sn}-\ref{tab:main42}, as well as highlighting some of the specific cases that arise.
\begin{itemize}\addtolength{\itemsep}{0.2\baselineskip}
\item[(a)] In each table, the given subgroup $H$ is recorded up to conjugacy in $G$, and we indicate if there are two or more classes of subgroups of the same type. For example, $G = A_8$  has two classes of subgroups $H$ isomorphic to ${\rm AGL}_{3}(2)$ and in both cases $d_G(H) = 7$, as recorded in Table \ref{tab:sn}.

\item[(b)] In the final column of Table \ref{tab:sn}, we indicate whether or not $H$ is maximal in $G$. And if $H$ is not maximal, we present a chain of subgroups $H = H_k < H_{k-1} < \cdots < H_1$, where each $H_i$ is maximal in $H_{i-1}$, and $H_1$ is maximal in $G$. The same set-up is also adopted in Tables \ref{tab:main41} and \ref{tab:main42}.

\item[(c)] In Table \ref{tab:main2}, each subgroup $H$ in the third column is a $2$-maximal subgroup of $G$, and in the final column of the table we record a maximal subgroup $H_1$ of $G$, which  contains $H$ as a maximal subgroup. Similarly, the subgroups $H$ appearing in Table \ref{tab:main3} are $3$-maximal and we record subgroups $H_2$ and $H_1$ such that $H < H_2 < H_1$ is a maximal chain, with $H_1$ maximal in $G$ (in general, such a maximal chain is not unique).

\item[(d)] Consider the case $G = {\rm Co}_3$ in Table \ref{tab:main2}. The maximal subgroup ${\rm U}_4(3).2^2$ has three index-two subgroups, namely ${\rm U}_4(3).2_1$ and two copies of ${\rm U}_4(3).2_3$ (which are non-conjugate in $G$).
Here ${\rm U}_4(3).2_1$ has depth $5$ in $G$, whereas one copy of ${\rm U}_4(3).2_3$ has depth $5$ and the other has depth $4$. As before, our notation ${\rm U}_4(3).2_i$ is consistent with the Atlas \cite{Atlas}, which means that $H = {\rm U}_4(3).2_3$ contains an involutory graph automorphism $x$ with $C_{G_0}(x) = {\rm PO}_{4}^{-}(3)$.

\item[(e)] The group $G = {\rm Co}_2$ has two conjugacy classes of subgroups of the form $2^9.{\rm L}_3(4).2$, represented by $K_1$ and $K_2$. Here $K_1 \cong K_2$ and $d_G(K_i) = 5$ for $i=1,2$. In addition, $K_1 = P_3 < {\rm U}_6(2).2$ is $2$-maximal (so it is listed in Table \ref{tab:main2}), whereas $K_2 < 2^{10}.{\rm L}_3(4).2 < 2^{10}.{\rm M}_{22}.2$ is $3$-maximal and therefore appears in Table \ref{tab:main3}.

\item[(f)] Suppose $G = {\rm Fi}_{23}$. As indicated in Table \ref{tab:main3}, $G$ has a $3$-maximal subgroup $H = \O_7(3) \times 2$ of depth $5$, which is embedded in $G$ via the following maximal chain:
\[
\O_7(3) \times 2 < {\rm P\O}_{8}^{+}(3).2 < {\rm P\O}_{8}^{+}(3).S_3 < {\rm Fi}_{23}.
\]
Here the $\O_7(3)$ subgroup of $H$ has depth $3$ and we note that $H$ is not conjugate to the $2$-maximal subgroup $K = \O_7(3) \times 2 < \O_7(3) \times S_3$. Indeed, in the latter case, the $\O_7(3)$ subgroup of $K$ has depth $5$.

\item[(g)] Let $G = {\rm Fi}_{24}'$. The maximal subgroup $(3 \times {\rm P\O}_{8}^{+}(3).3).2$ has three conjugacy classes of maximal subgroups isomorphic to ${\rm P\O}_{8}^{+}(3).S_3$; two have depth $3$, while the third has depth $5$. The index-two subgroup ${\rm P\O}_{8}^{+}(3).3$ of the latter also has depth $5$, and it can also be viewed as a maximal subgroup of $3 \times {\rm P\O}_{8}^{+}(3).3$.
\end{itemize}
\end{rem}

{\scriptsize
\begin{table}[h]
\[
\begin{array}{clll} \hline
d & G & H &  \\ \hline
11 & {\rm M}_{24} & {\rm M}_{23} & \\ 
& & & \\
9 & {\rm M}_{23} & {\rm M}_{22} & \\
& {\rm Co}_3 & {\rm McL}.2 & \\
& {\rm Co}_2 & {\rm U}_6(2).2 & \\
& {\rm Fi}_{23} & 2.{\rm Fi}_{22} & \\
& & & \\
7 & {\rm M}_{11} & {\rm M}_{10} = A_6.2 & \\
& {\rm M}_{12} & {\rm M}_{11} \, \mbox{(two)} & \\
& {\rm M}_{22} & {\rm L}_3(4) & \\
& {\rm He} & {\rm Sp}_4(4).2 & \\
& {\rm Suz} & G_2(4) & \\
& {\rm HS} & {\rm M}_{22},\, {\rm U}_3(5).2 \, \mbox{(two)} & \\
& {\rm Co}_1 & {\rm Co}_2 & \\
& {\rm Fi}_{22} & 2.{\rm U}_6(2),\, \O_7(3) \, \mbox{(two)} & \\
& {\rm Fi}_{23} & {\rm P\O}_8^{+}(3).S_3 & \\
& {\rm Fi}_{24}' & {\rm Fi}_{23} & \\
& \mathbb{B} & 2.{}^2E_6(2).2 & \\
& & & \\
5 & {\rm M}_{11} & {\rm L}_2(11), \, {\rm M}_9{:}2 = {\rm U}_3(2){:}2  & \\
& {\rm M}_{12} & A_6.2^2 \, \mbox{(two)}, \, {\rm L}_2(11), \, {\rm AGL}_2(3) \, \mbox{(two)}, \, 2 \times S_5 &  \\
& {\rm M}_{22} & 2^4{:}A_6,\, A_7 \, \mbox{(two)}, \, 2^4{:}S_5,\, {\rm AGL}_3(2), \, {\rm L}_2(11) & \\
& {\rm M}_{23} & {\rm L}_3(4).2, \, 2^4{:}A_7, \,  A_8, \, {\rm M}_{11}, \, 2^4{:}(3 \times A_5){:}2 & \\
& {\rm M}_{24} & {\rm M}_{22}.2, \, 2^4.A_8,\, {\rm M}_{12}.2, \, 2^6{:}3.S_6,\, {\rm L}_3(4){:}S_3,\, 2^6{:}({\rm L}_3(2) \times S_3) & \\
& {\rm J}_2 & {\rm U}_3(3),\, 3.A_6.2,\, 2^{1+4}.A_5,\, 2^{2+4}{:}(3 \times S_3),\, A_5 \times D_{10} & \\

& {\rm He} & 2^2.{\rm L}_3(4).S_3, \, 2^6{:}3.S_6 \, \mbox{(two)} & \\

& {\rm McL} & {\rm U}_4(3), \,  {\rm M}_{22} \,\mbox{(two)},\, {\rm U}_3(5),\, 3^{1+4}{:}2.S_5,\, 3^4{:}{\rm M}_{10} & \\
& {\rm Suz} & 3.{\rm U}_4(3).2,\, {\rm U}_5(2),\, 2^{1+6}.{\rm U}_4(2),\, 3^5.{\rm M}_{11}, \, 2^{4+6}{:}3A_6 & \\
& {\rm HS} & S_8,\, 2^4.S_6,\, 4^3{:}{\rm L}_3(2),\, 4.2^4.S_5 & \\
& {\rm Ru} & {}^2F_4(2)'.2,\, 2^6.{\rm U}_3(3).2 &  \\
& {\rm Co}_3 & {\rm HS},\, {\rm U}_4(3).2^2,\, {\rm M}_{23}, \, 3^5{:}(2 \times {\rm M}_{11}),\, 2.{\rm Sp}_6(2),\, 3^{1+4}{:}4S_6 & \\
& {\rm Co}_{2} & 2^{10}{:}{\rm M}_{22}.2, \, {\rm McL},\, 2^{1+8}.{\rm Sp}_6(2),\, {\rm HS}.2, \, (2^4 \times 2^{1+6}).A_8,\, {\rm U}_4(3).D_8,\, 2^{4+10}.(S_5 \times S_3) & \\
& {\rm Co}_1 & 3.{\rm Suz}.2, \, 2^{11}.{\rm M}_{24},\, {\rm Co}_3, \, 2^{1+8}.\O_8^{+}(2),\, {\rm U}_6(2).S_3 & \\
& {\rm Fi}_{22} & \O_8^{+}(2).S_3,\, 2^{10}.{\rm M}_{22},\, 2^6.{\rm Sp}_6(2),\, 2.2^{1+8}.{\rm U}_4(2).2, \, {\rm U}_4(3).2 \times S_3,\, 2^{5+8}.(S_3 \times A_6) & \\
& {\rm Fi}_{23} & 2^2.{\rm U}_6(2).2, \, {\rm Sp}_8(2), \, \O_7(3) \times S_3, \, 2^{11}.{\rm M}_{23}, \, 3^{1+8}.2^{1+6}.3^{1+2}.2S_4 & \\
& {\rm O'N} & {\rm L}_3(7).2 \, \mbox{(two)} & \\
& {\rm Ly} & 3.{\rm McL}.2 & \\
& {\rm HN} & A_{12}, \, 2.{\rm HS}.2, \, {\rm U}_3(8).3 & \\
& {\rm Fi}_{24}' & (3 \times {\rm P\O}_{8}^{+}(3).3).2, \, 2.{\rm Fi}_{22}.2 & \\
& \mathbb{B} & 2^{1+22}.{\rm Co}_2,\, 2^{9+16}.{\rm Sp}_8(2) & \\
& \mathbb{M} & 2.\mathbb{B} & \\
& & & \\
4 & {\rm HS} & {\rm L}_3(4).2 & \\
& {\rm Fi}_{24}' & 3^{1+10}.{\rm U}_5(2).2 & \\ \hline
\end{array}
\]
\caption{The pairs $(G,H)$ in Theorem \ref{t:main0} with $d = d_G(H) > 3$}
\label{tab:max}
\end{table}
}

{\scriptsize
\begin{table}[h]
\[
\begin{array}{clll} \hline
d & G & H & H_1 \\ \hline
7 & {\rm Co}_3 & {\rm McL} & {\rm McL}.2 \\
& {\rm Co}_2 & {\rm U}_6(2) & {\rm U}_6(2).2 \\
& {\rm Fi}_{23} & {\rm P\O}_8^{+}(3).3 & {\rm P\O}_8^{+}(3).S_3 \\
& & & \\
5 & {\rm M}_{11} & A_6, \, S_3 \wr S_2 & {\rm M}_{10}, \, {\rm M}_9{:}2 \\
& {\rm M}_{12} & S_6 \, \mbox{(two)},\,  A_6.2 = {\rm PGL}_2(9) \, \mbox{(two)},\,A_6.2 = {\rm M}_{10} \, \mbox{(two)} & A_6.2^2 \\
& {\rm M}_{22} & A_4^2.4 & 2^4{:}A_6 \\
& {\rm M}_{23} & {\rm L}_3(4) & {\rm L}_3(4).2 \\
& {\rm M}_{24} & 2^6{:}3.A_6,\, 2^6{:}3.S_5 \, \mbox{(two)},\, 2^6{:}3.(S_3 \wr S_2), \, 2^6{:}3.(S_4 \times S_2) & 2^6{:}3.S_6 \\
& & 2^4.A_7,\, 2^4.{\rm AGL}_3(2)  & 2^4.A_8  \\
& & 2^6{:}({\rm L}_3(2) \times 3),\, 2^6{:}({\rm L}_3(2) \times 2) & 2^6{:}({\rm L}_3(2) \times S_3) \\
&& {\rm M}_{22}, \, {\rm M}_{12} & {\rm M}_{22}.2,\, {\rm M}_{12}.2  \\
& {\rm J}_2 & 3.A_6,\, 2^{2+4}.3^2 & 3.A_6.2, \, 2^{2+4}{:}(3 \times S_3) \\
& {\rm He} & {\rm Sp}_4(4) & {\rm Sp}_4(4).2 \\
& {\rm McL} & 3^4.A_6,\, 3^{1+4}{:}2.A_5 & 3^4.{\rm M}_{10},\, 3^{1+4}{:}2.S_5 \\
& {\rm Suz} & 3.{\rm U}_4(3) & 3.{\rm U}_4(3).2 \\
& {\rm HS} & {\rm U}_3(5) \, \mbox{(two)}, \, S_7 & {\rm U}_3(5).2,\, S_8 \\
& {\rm Ru} & {}^2F_4(2)' & {}^2F_4(2)'.2 \\
& {\rm Co}_3 & 3^5.{\rm M}_{11},\, {\rm U}_4(3).2_1,\, {\rm U}_4(3).2_3 & 3^5{:}(2 \times {\rm M}_{11}), \, {\rm U}_4(3).2^2,\, {\rm U}_4(3).2^2 \\
& {\rm Co}_{2} & {\rm U}_5(2).2,\, P_1 = 2^{1+8}.{\rm U}_4(2).2,\, P_3 = 2^9.{\rm L}_3(4).2 & {\rm U}_6(2).2 \\
& & 2^{4+10}.(S_5 \times 3), \, 2^{4+10}.(S_5 \times 2), \, 2^{4+8}.(S_5 \times S_3) & 2^{4+10}.(S_5 \times S_3) \\
& & 2^{10}{:}{\rm M}_{22}, \, 2^{10}.{\rm L}_3(4).2 & 2^{10}{:}{\rm M}_{22}.2 \\
& & 2^{1+8}.2^5.{\rm Sp}_4(2),\, 2^{1+8}.S_8 & 2^{1+8}.{\rm Sp}_6(2) \\
& & {\rm U}_4(3).2^2 & {\rm U}_4(3).D_8 \\
& {\rm Co}_1 & 3.{\rm Suz},\, 2^{11}.{\rm M}_{23} & 3.{\rm Suz}.2, \, 2^{11}.{\rm M}_{24} \\
& & {\rm U}_6(2).3, \, {\rm U}_6(2).2 & {\rm U}_6(2).S_3 \\
& {\rm Fi}_{22} & \O_8^{+}(2).3, \, \O_8^{+}(2).2 & \O_8^{+}(2).S_3 \\
& & {\rm U}_4(3).2 \times 3,\, {\rm U}_4(3).2 \times 2 & {\rm U}_4(3).2 \times S_3 \\
& & 2 \times {\rm U}_5(2),\, 2.2^{1+8}.{\rm U}_4(2) & 2.{\rm U}_6(2),\, 2.2^{1+8}.{\rm U}_4(2).2 \\
& {\rm Fi}_{23} & 2^2.{\rm U}_6(2),\, 2 \times \O_8^{+}(2).S_3 & 2^2.{\rm U}_6(2).2 \\
& & {\rm P\O}_8^{+}(3).2 & {\rm P\O}_8^{+}(3).S_3 \\
& & \mbox{$3^{1+8}.2^{1+6}.3^{1+2}.J$ with $J = D_{12}$, $SD_{16}$, ${\rm SL}_2(3)$} & 3^{1+8}.2^{1+6}.3^{1+2}.2S_4 \\
& & \O_7(3) \times 3, \, \O_7(3) \times 2 & \O_7(3) \times S_3 \\
& {\rm Ly} & 3.{\rm McL} & 3.{\rm McL}.2 \\
& {\rm Fi}_{24}' & {\rm P\O}_{8}^{+}(3).S_3,\, 3 \times {\rm P\O}_{8}^{+}(3).3 & (3 \times {\rm P\O}_{8}^{+}(3).3).2 \\
& \mathbb{B} & 2.{}^2E_6(2) & 2.{}^2E_6(2).2 \\
& & & \\
4 & {\rm Co}_3 & {\rm U}_4(3).2_{3} & {\rm U}_4(3).2^2 \\ 
& {\rm Co}_{2} & 2^{1+8}.{\rm U}_3(3).2 & 2^{1+8}.{\rm Sp}_6(2)\\ 
& {\rm Fi}_{24}' & 3.{\rm P\O}_{8}^{+}(3).2 & (3 \times {\rm P\O}_{8}^{+}(3).3).2 \\ \hline
\end{array}
\]
\caption{The pairs $(G,H)$ in part (ii) of Theorem \ref{t:main}}
\label{tab:main2}
\end{table}
}

{\scriptsize
\begin{table}[h]
\[
\begin{array}{llll} \hline
G & H & H_2 & H_1 \\ \hline
{\rm M}_{24} & 2^6{:}{\rm L}_3(2) & 2^6{:}({\rm L}_3(2) \times 3) & 2^6{:}({\rm L}_3(2) \times S_3) \\
{\rm Co}_{2} & 2^{1+8}.{\rm U}_4(2) & P_1 & {\rm U}_6(2).2 \\
 & 2^{9}.{\rm L}_3(4).2 & 2^{10}.{\rm L}_3(4).2 & 2^{10}{:}{\rm M}_{22}.2 \\
{\rm Co}_1 & {\rm U}_6(2) & {\rm U}_6(2).2 & {\rm U}_6(2).S_3 \\
{\rm Fi}_{22} & {\rm U}_4(3).2 & {\rm U}_4(3).2 \times 2 & {\rm U}_4(3).2 \times S_3 \\
& \O_{8}^{+}(2) & \O_8^{+}(2).2 & \O_8^{+}(2).S_3 \\
{\rm Fi}_{23} & {\rm P\O}_8^{+}(3) & {\rm P\O}_8^{+}(3).2 & {\rm P\O}_8^{+}(3).S_3 \\
& 3^6.{\rm L}_3(4).2 & {\rm P\O}_8^{+}(3).2 & {\rm P\O}_8^{+}(3).S_3 \\
& \O_7(3) \times 2 & {\rm P\O}_8^{+}(3).2 & {\rm P\O}_8^{+}(3).S_3 \\
& \O_7(3) & \O_7(3) \times 2 & \O_7(3) \times S_3 \\
& 3^{1+8}.2^{1+6}.3^{1+2}.D_8 & 3^{1+8}.2^{1+6}.3^{1+2}.SD_{16} &  3^{1+8}.2^{1+6}.3^{1+2}.2S_4 \\ 
{\rm Fi}_{24}' & {\rm P\O}_{8}^{+}(3).3 & {\rm P\O}_{8}^{+}(3).S_3 & (3 \times {\rm P\O}_{8}^{+}(3).3).2 \\
\hline
\end{array}
\]
\caption{The pairs $(G,H)$ in part (iii) of Theorem \ref{t:main} with $d_G(H) = 5$}
\label{tab:main3}
\end{table}
}
 
{\scriptsize
\begin{table}[h]
\[
\begin{array}{clll} \hline
d & G & H &  \\ \hline
9 & {\rm M}_{22}.2 & {\rm L}_3(4).2 & \mbox{maximal} \\
& {\rm HS}.2 & {\rm M}_{22}.2 & \mbox{maximal} \\
& {\rm Fi}_{22}.2 & 2.{\rm U}_6(2).2 & \mbox{maximal} \\
& {\rm Fi}_{24}'.2 & {\rm Fi}_{23} \times 2 & \mbox{maximal} \\
& & & \\
7 & {\rm M}_{12}.2 & {\rm M}_{11} & H < G_0 \\
& {\rm M}_{22}.2 & 2^4{:}S_6 & \mbox{maximal} \\
& & {\rm L}_3(4) & H<G_0 \\
& {\rm HS}.2 & {\rm M}_{22}, \, {\rm U}_3(5).2 & \mbox{maximal} \\
& {\rm He}.2 & {\rm Sp}_4(4){:}4 & \mbox{maximal} \\
& & {\rm Sp}_4(4){:}2 & H<G_0 \\
& {\rm McL}.2 & {\rm U}_4(3).2 & \mbox{maximal} \\
& {\rm Suz}.2 & G_2(4).2 & \mbox{maximal} \\
& & G_2(4) & H<G_0 \\
& {\rm Fi}_{22}.2 & \O_8^{+}(2){:}S_3 \times 2,\, 2^{10}{:}{\rm M}_{22}.2 & \mbox{maximal} \\
& & 2.{\rm U}_6(2),\, \O_7(3) & H<G_0 \\
& {\rm Fi}_{24}'.2 & {\rm Fi}_{23} & H < G_0 \\
& & & \\
5 & {\rm M}_{12}.2 & {\rm L}_2(11).2 \, \mbox{(two)},\, (2^2 \times A_5){:}2,\, 2^{1+4}{:}S_3.2,\, 4^2{:}D_{12}.2 & \mbox{maximal} \\
& & A_6.2^2,\, {\rm L}_2(11),\, {\rm AGL}_2(3), \, 2 \times S_5 & H <G_0 \\
& & S_6,\, {\rm PGL}_2(9) & H < A_6.2^2 < G_0 \\ 
& {\rm M}_{22}.2 & 2^5{:}S_5,\, 2^3{:}{\rm L}_3(2) \times 2, \, A_6.2^2, \, {\rm L}_2(11).2 & \mbox{maximal} \\
& & 2^4{:}A_6,\, A_7,\, 2^4{:}S_5,\, 2^3{:}{\rm L}_3(2) & H<G_0 \\
& & 2^4.S_5,\, S_6, \, A_4^2.D_8,\, 4^2.A_4.2^2,\, 2^3.A_4.2^3 & H < 2^4{:}S_6 \\
& & 2^5{:}A_5,\, 2^4.S_5 \, \mbox{(two)} & H < 2^5{:}S_5 \\
& & S_6 & H < A_6.2^2 \\
& & A_4^2.4 & H < 2^4{:}A_6 < G_0 \\
& {\rm J}_2.2 & {\rm U}_3(3).2,\, 3.A_6.2^2, \, 2^{1+4}.S_5,\, 2^{2+4}{:}(3 \times S_3).2 & \mbox{maximal} \\
& & (A_4 \times A_5).2, \, (A_5 \times D_{10}).2 & \mbox{maximal} \\
& & {\rm U}_3(3),\, 3.A_6.2, \, 2^{1+4}{:}A_5,\, 2^{2+4}{:}(3 \times S_3),\, A_5 \times D_{10} & H<G_0 \\
& & 2^4.A_4.6, \, 4^2.A_4.S_3,\, 4^2.A_4.2^2 & H < 2^{2+4}{:}(3 \times S_3).2 \\
& & 2^2.A_4^2 & H < 2^4.A_4.6 < 2^{2+4}{:}(3 \times S_3).2 \\
& & 3.A_6.2 \, \mbox{(two)} & H < 3.A_6.2^2 \\
& & 3.A_6 & H < 3.A_6.2 < 3.A_6.2^2 \\
& {\rm HS}.2 & {\rm L}_3(4).2^2,\,S_8 \times 2,\, 2^5.S_6,\, 4^3{:}(3 \times {\rm L}_3(2)),\, 2^{1+6}.S_6 & \mbox{maximal} \\ 
& & (2 \times A_6.2^2).2 & \mbox{maximal} \\
& & 2^4.S_6,\, 4^3{:}{\rm L}_3(2),\, 4.2^4.S_5  & H<G_0 \\ 
& & 2^4.S_6 & H < 2^5.S_6 \\
& & S_7 \times 2,\, S_8 \, \mbox{(two)},\, A_8 \times 2 & H < S_8 \times 2 \\
& & {\rm L}_3(4).2 & H < {\rm L}_3(4).2^2 \\
& & S_7 & H < S_7 \times 2 < S_8 \times 2 \\
& & {\rm U}_3(5) & H < {\rm U}_3(5).2 < G_0 \\
& {\rm He}.2 & 2^2.{\rm L}_3(4).D_{12} & \mbox{maximal} \\
& & 2^2.{\rm L}_3(4).S_3,\, 2^6{:}3.S_6 & H<G_0 \\
& & 2^2.{\rm L}_3(4).6,\, 2^2.{\rm L}_3(4).2^2 & H < 2^2.{\rm L}_3(4).D_{12} \\
& & {\rm Sp}_4(4) & H < {\rm Sp}_4(4){:}2 < G_0 \\ 
& {\rm McL}.2 & {\rm U}_3(5).2,\, 3^{1+4}{:}4.S_5,\, 3^4{:}({\rm M}_{10} \times 2),\, 2.S_8 & \mbox{maximal} \\
& & {\rm U}_4(3),\, {\rm M}_{22},\, {\rm U}_3(5),\, 3^{1+4}{:}2.S_5,\, 3^4{:}{\rm M}_{10} & H<G_0 \\
& & 3^{1+4}{:}4.A_5,\, 3^{1+4}.2.S_5  & H < 3^{1+4}{:}4.S_5 \\
& & 3^4.{\rm M}_{10} & H < {\rm U}_4(3){:}2 \\
& & 3^4.(A_6 \times 2) & H < 3^4{:}({\rm M}_{10} \times 2) \\
& & 3^{1+4}{:}2A_5 & H < 3^{1+4}{:}2S_5 < G_0 \\
& & 3^4.A_6 & H < 3^4{:}{\rm M}_{10} < G_0 \\
& {\rm Suz}.2 & 3.{\rm U}_4(3).2^2, \, {\rm U}_5(2).2, \, 2^{1+6}.{\rm U}_4(2).2, \, 3^5{:}({\rm M}_{11} \times 2) & \mbox{maximal} \\
& & {\rm J}_2{:}2 \times 2, \, 2^{4+6}{:}3S_6 & \mbox{maximal} \\ 
& & 3.{\rm U}_4(3).2, \, {\rm U}_5(2), \, 2^{1+6}.{\rm U}_4(2), \, 3^5{:}{\rm M}_{11},\, 2^{4+6}{:}3A_6 & H<G_0 \\
& & 3.{\rm U}_4(3).2 \, \mbox{(two)} & H < 3.{\rm U}_4(3).2^2 \\
& & 3.{\rm U}_4(3) & H < 3.{\rm U}_4(3).2 < G_0 \\ \hline
\end{array}
\]
\caption{The pairs $(G,H)$ in Theorem \ref{t:main2}(iii), Part I}
\label{tab:main41}
\end{table}
}

{\scriptsize
\begin{table}[h]
\[
\begin{array}{clll} \hline
d & G & H &  \\ \hline
5 & {\rm Fi}_{22}.2 & 2^7{:}{\rm Sp}_6(2),\, (2 \times 2^{1+8}{:}{\rm U}_4(2){:}2){:}2,\, S_3 \times {\rm U}_4(3).2^2 & \mbox{maximal}\\
& & 2^{5+8}{:}(S_3 \times S_6),\, 3^5{:}(2 \times {\rm U}_4(2){:}2) & \mbox{maximal} \\
& & 3^{1+6}{:}2^{3+4}{:}3^2{:}2.2 & \mbox{maximal} \\
& & 2^2 \times {\rm Sp}_6(2),\, 2 \times {\rm U}_5(2).2 & H < 2.{\rm U}_6(2).2 \\
& & \O_8^{+}(2).3 \times 2,\, \O_8^{+}(2).2 \times 2,\, \O_8^{+}(2).S_3 \, \mbox{(two)} & H < \O_8^{+}(2).S_3 \times 2 \\
& & 2^{10}.{\rm L}_3(4).2,\, 2^{10}{:}{\rm M}_{22} & H < 2^{10}{:}{\rm M}_{22}{:}2 \\
& & 2^6{:}{\rm Sp}_6(2) & H < 2^7{:}{\rm Sp}_6(2) \\
& & 2.2^8.3^3.A_4.2^4,\, 2^2.2^8.2.{\rm U}_4(2),\, 2^2.2^8.{\rm U}_4(2).2 \, \mbox{(two)} & H < (2 \times 2^{1+8}{:}{\rm U}_4(2){:}2){:}2 \\
& & S_3 \times {\rm U}_4(3).2\, \mbox{(three)},\, 3 \times {\rm U}_4(3).2^2,\, 2 \times {\rm U}_4(3).2^2  & H < S_3 \times {\rm U}_4(3).2^2 \\
& & 3.{\rm U}_4(3).2^2 \, \mbox{(three)}  & H < S_3 \times {\rm U}_4(3).2^2 \\
& & 2^{5+8}{:}(3 \times S_6),\, 2^{5+8}{:}(2 \times S_6),\, 2^{5+8}{:}(S_3 \times A_6) & H < 2^{5+8}{:}(S_3 \times S_6) \\
& & 3^5{:}{\rm U}_4(2).2 & H < 3^5{:}(2 \times {\rm U}_4(2).2) \\
& & {\rm L}_4(3).2 & H < \O_7(3) < G_0 \\
& & 2 \times {\rm U}_5(2),\, {\rm U}_5(2).2 \, \mbox{(two)} & H < 2 \times {\rm U}_5(2).2 < 2.{\rm U}_6(2).2 \\
& & 2.{\rm U}_4(3).2 & H < 2.{\rm U}_6(2) < G_0 \\
& & \O_8^{+}(2).2 \, \mbox{(three)} & H < \O_8^{+}(2).2 \times 2 < \O_8^{+}(2).S_3 \times 2 \\ 
& & \O_8^{+}(2).3 & H < \O_8^{+}(2).S_3 < \O_8^{+}(2).S_3 \times 2 \\
& & 2^{2+8}.{\rm U}_4(2) & H < 2^{2+8}.{\rm U}_4(2).2 < (2 \times 2^{1+8}{:}{\rm U}_4(2){:}2){:}2 \\
& & 2 \times {\rm U}_4(3).2,\, {\rm U}_4(3).2^2 \, \mbox{(four)} & H < 2 \times {\rm U}_4(3).2^2 < S_3 \times {\rm U}_4(3).2^2 \\
& & 3 \times {\rm U}_4(3).2 & H < S_3 \times {\rm U}_4(3).2 < S_3 \times {\rm U}_4(3).2^2 \\ 
& & 2^{5+8}.S_6 & H < 2^{5+8}{:}(3 \times S_6) < 2^{5+8}{:}(S_3 \times S_6) \\
& & \O_8^{+}(2) & H < \O_8^{+}(2) \times 2 < \O_8^{+}(2).2 \times 2 < \O_8^{+}(2).S_3 \times 2 \\
& & {\rm U}_4(3).2_2 & H < 2 \times {\rm U}_4(3).2 < 2 \times {\rm U}_4(3).2^2 < S_3 \times {\rm U}_4(3).2^2 \\
& {\rm O'N}.2 & {\rm L}_3(7).2 & H < G_0 \\
& {\rm HN}.2 & S_{12},\, 4.{\rm HS}.2, \, {\rm U}_3(8).6 & \mbox{maximal} \\
& & A_{12},\, 2.{\rm HS}.2,\, {\rm U}_3(8).3 & H < G_0 \\
& & 4.{\rm HS} & H < 4.{\rm HS}.2 \\
& {\rm Fi}_{24}'.2 & (2 \times 2.{\rm Fi}_{22}).2,\, S_3 \times {\rm P\O}_8^{+}(3).S_3,\, {\rm O}_{10}^{-}(2) & \mbox{maximal} \\
& & 3^7.\O_7(3).2,\, 3^{1+10}.({\rm U}_5(2).2 \times 2) & \mbox{maximal} \\
& & 2.{\rm Fi}_{22}.2 \, \mbox{(two)},\, 2 \times 2.{\rm Fi}_{22} & H < (2 \times 2.{\rm Fi}_{22}).2 \\
& & 3.{\rm P\O}_8^{+}(3).S_3 \, \mbox{(two)}, \, 2.{\rm P\O}_8^{+}(3).S_3 & H < S_3 \times {\rm P\O}_8^{+}(3).S_3 \\
& & S_3.{\rm P\O}_8^{+}(3).3, \, S_3.{\rm P\O}_8^{+}(3).2 & H < S_3 \times {\rm P\O}_8^{+}(3).S_3 \\
& & 3^{1+10}.({\rm U}_5(2) \times 2),\, 3^{1+10}.{\rm U}_5(2).2 & H < 3^{1+10}.({\rm U}_5(2).2 \times 2) \\
& & 2.{\rm Fi}_{22} & H < 2 \times 2.{\rm Fi}_{22} < (2 \times 2.{\rm Fi}_{22}).2 \\ 
& & 3.{\rm P\O}_8^{+}(3).2,\, S_3.{\rm P\O}_8^{+}(3) & H <  S_3.{\rm P\O}_8^{+}(3).2 < S_3 \times {\rm P\O}_8^{+}(3).S_3 \\
& & {\rm P\O}_8^{+}(3).S_3 \, \mbox{(two)},\, 2.{\rm P\O}_8^{+}(3).3 & H < 2.{\rm P\O}_8^{+}(3).S_3 < S_3 \times {\rm P\O}_8^{+}(3).S_3 \\
& & 3.{\rm P\O}_8^{+}(3).3 & H < 3.{\rm P\O}_8^{+}(3).S_3 < S_3 \times {\rm P\O}_8^{+}(3).S_3 \\ 
& & {\rm P\O}_8^{+}(3).3 & H < 2.{\rm P\O}_8^{+}(3).3 < 2.{\rm P\O}_8^{+}(3).S_3 < S_3 \times {\rm P\O}_8^{+}(3).S_3 \\ 
& & & \\
4 & {\rm HS}.2 & {\rm L}_3(4).2 & H<G_0 \\ 
& {\rm Fi}_{22}.2 & 2 \times {\rm U}_4(3).2 \, \mbox{(two)} & H < 2 \times {\rm U}_4(3).2^2 < S_3 \times {\rm U}_4(3).2^2 \\
& & 3 \times {\rm U}_4(3).2 \, \mbox{(two)} & H < 3 \times {\rm U}_4(3).2^2 < S_3 \times {\rm U}_4(3).2^2 \\ 
& {\rm Fi}_{24}'.2 & 3^{1+10}.{\rm U}_5(2).2 & H < 3^{1+10}.({\rm U}_5(2).2 \times 2) \\ 
& & 3.{\rm P\O}_8^{+}(3).2 & H <  S_3.{\rm P\O}_8^{+}(3).2 < S_3 \times {\rm P\O}_8^{+}(3).S_3 \\ \hline
\end{array}
\]
\caption{The pairs $(G,H)$ in Theorem \ref{t:main2}(iii), Part II}
\label{tab:main42}
\end{table}
}

\clearpage

Conflicts of interest: none.

\end{document}